\documentclass[12pt]{article} 

\usepackage{latexsym,amssymb,amsthm,amsmath}
\usepackage{url,graphics}
\usepackage[dvips]{graphicx,epsfig}

\usepackage[T2A]{fontenc}
\usepackage[utf8]{inputenc}
\usepackage{amsfonts}
\usepackage{amsthm}
\usepackage{amsmath}
\usepackage{amssymb}
\usepackage{hyperref}
\usepackage{pdfpages}
\usepackage{dsfont}
\usepackage{scalerel}
\usepackage{makecell,xcolor}
\usepackage{tikz}
\usepackage[12pt]{extsizes} 
\usepackage{graphicx} 
\graphicspath{{pictures/}} 
\DeclareGraphicsExtensions{.pdf} 
\DeclareGraphicsExtensions{.png} 
\usepackage[utf8]{inputenc} 
\usepackage[russian, english]{babel}
\usepackage{amsmath} 
\usepackage{amsthm, amsmath, amssymb} 
\usepackage{tikz} 
\usepackage{amsfonts}
\usepackage[left=2cm,right=2cm,top=2cm,bottom=2cm,bindingoffset=0cm]{geometry} 
\usepackage{wrapfig} 
\usepackage{graphicx}
\graphicspath{{pictures/}} 
\DeclareGraphicsExtensions{.pdf} 
\DeclareGraphicsExtensions{.png} 
\usepackage{caption}

\graphicspath{[pictures/]} 
\DeclareGraphicsExtensions{.pdf,.png,.jpg}

\renewcommand{\geq}{\geqslant}

\renewcommand{\leq}{\leqslant}

\title{\textbf{An existence criterion for a cycle such that the vertex set beyond this cycle is independent}}

\author{Nikolai Karol}
\date{2020}

\begin{document}

\maketitle

\newtheorem{thm}{Theorem}
\newtheorem{lem}{Lemma}
\renewcommand*{\proofname}{\bf Proof}
\newtheorem{cor}{Corollary}
\newtheorem*{conj}{Conjecture}
\newtheorem{claim}{Claim}
\newtheorem{subcl}{Subclaim}[claim]
\theoremstyle{definition}
\newtheorem{defin}{Definition}
\theoremstyle{remark}
\newtheorem{rem}{\bf Remark}
\theoremstyle{configuration}
\newtheorem{conf}{\bf Configuration}

\def\mmax{\mathop{\rm max}}
\def\q#1.{{\bf #1.}}
\def\P{{\rm Part}}
\def\QP{{\rm QPart}}
\def\I{{\rm Int}}
\def\R{{\rm Bound}}
\def\RC{{\rm Cut}}
\def\B{{\rm BT}}
\def\T{{\rm T}}
\def\N{{\rm N}}
\def\GM{{\cal GM}}

\begin{center}
    \textbf{Abstract}
\end{center}

{\small We prove that if $G$ is a 2-connected graph with $\delta(G) \geqslant \frac{v(G) + 2}{3}$ then $G$ has a cycle $W$ such that $V(G - W)$ is independent. This result is best possible in the sense that it becomes false if $\frac{v(G) + 2}{3}$ is replaced by any smaller number.}

\vspace{8.5mm}

\textbf{\Large{1 Introduction}}

$ $

The Hamiltonian cycle existence problem is well-known and well explored. Apart from classical Dirac and Ore's criteria, Chvátal's theorems (see~\cite{bondy1}) are worthy of mention. In addition to these results, there is a problem whether the graph power is Hamiltonian. There are classical results here: well-known Fleischner's theorem (see~\cite{diestel2}) and Chartrand and Kapoor's theorem of~\cite{chartrandkapoor}.

As for long cycles, there is classical Linial's theorem of~\cite{linial3}.

$ $

\textbf{Theorem (N. Linial, 1975)} \textit{Let $G$ be a  2-connected graph and} 

$$m = \min\limits_{xy \notin E(G)} d_{G}(x) + d_{G}(y)\text{.}$$

\textit{Then the length of the longest cycle of G is at least min(m, v(G)).}

$ $

Thomassen's paper~\cite{thomassen4} contains an existence criterion in planar graphs for a cycle with special condition: the vertex set beyond it must be independent. Our work contains the following existence criterion for such a cycle in terms of graph's minimum degree.

\begin{thm} \label{mytheorem}  Let $G$ be a 2-connected graph, $v(G) = n$ and $\delta(G) \geqslant \frac{n + 2}{3}$. Then $G$ has a cycle such that the vertex set beyond this cycle is independent.
\end{thm}

%$ $
 
%\textbf{Theorem:} Let $G$ be a 2-connected graph, $v(G) = n$ and $\delta(G) \geqslant \frac{n + 2}{3}$. Then $G$ has such a cycle that the vertex set beyond this cycle is independent.

\vspace{8.5mm}

\textbf{\Large{2 Proof of Theorem~\ref{mytheorem}}}

$ $

We will invoke the following classical lemma in our proof, which is used in the proofs of Ore's theorem and several Chvátal's theorems.  

\begin{lem} \label{lemma}  Let $m > 2$ and $u_{1}...u_{m}$ be a path of the maximal length in the graph $G$, $d_{G}(u_{1}) + d_{G}(u_{m}) \geqslant m$. Then $G$ has a cycle of length m.
\end{lem}

%\textbf{Lemma 1:} Let $m > 2$, $u_{1}...u_{m}$ is a path of maximal length in graph $G$ and $d_{G}(u_{1}) + d_{G}(u_{m}) \geqslant m$. Then graph $G$ has a cycle of length m.

%\newpage

$ $

\paragraph{Proof of Theorem~\ref{mytheorem}.} For the sake of convenience, 5 Claims are highlighted.

We highlight that we will use $\delta(G) \geqslant \frac{n}{3}$. We will use $\delta(G) \geqslant \frac{n + 2}{3}$ in a couple of places only.

Assuming the converse, there exists graph $G$ such that $v(G) = n$, $\delta(G) \geqslant \frac{n + 2}{3}$, $G$ does not contain a cycle such that the vertex set beyond it is independent. 

Assume that $n < 6$. Since $G$ is 2-connected, $\delta(G) \geqslant 2$. If $n \in \{3, 4\}$ then it follows from $\delta(G) \geqslant 2$ that the graph contains a cycle. This cycle satisfies all the requirements because the vertex set beyond the cycle contains at most 1 vertex. If $n = 5$ then $G$ also has a cycle. If the cycle length is 4 or 5 then, similarly, this cycle meets all the requirements. Hence, the cycle length is 3. Let us denote vertices of the cycle by $r_{1}, r_{2}, r_{3}$ and the remaining vertices by $r_{4}, r_{5}$. If $r_{4}r_{5} \notin E(G)$ then cycle $r_{1}r_{2}r_{3}$ obviously meets all the conditions. Consequently, $r_{4}r_{5} \in E(G)$. Since $G$ is connected, $e_{G}(\{r_{1}, r_{2}, r_{3}\}, \{r_{4}, r_{5}\}) \geqslant 1$. Without loss of generality, $r_{1}r_{4} \in E(G)$. Since $G$ is 2-connected, $G - \{r_{1}\}$ is connected. Therefore, $e_{G}(\{r_{2}, r_{3}\}, \{r_{4}, r_{5}\}) \geqslant 1$. Without loss of generality, $r_{2}r_{4} \in E(G)$ or $r_{2}r_{5} \in E(G)$. If $r_{2}r_{4} \in E(G)$ then $r_{1}r_{4}r_{2}r_{3}$ is a suitable cycle (see figure~\ref{figureone} a). If $r_{2}r_{5} \in E(G)$ then $r_{1}r_{4}r_{5}r_{2}r_{3}$ is a suitable Hamiltonian cycle (see figure~\ref{figureone} b).

\begin{figure}[htb]
	\centering
	\includegraphics[width=0.6\columnwidth, keepaspectratio]{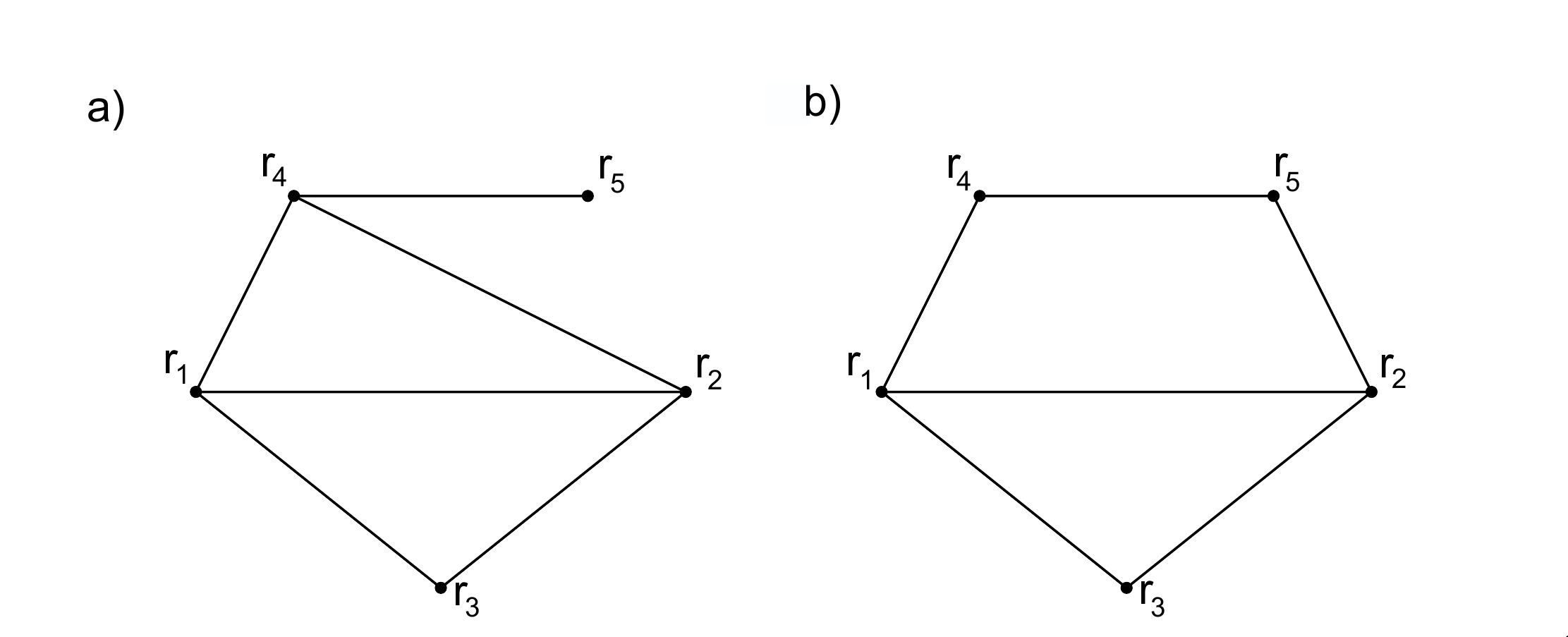}
	\caption{$n = 5$.}
	\label{figureone}
\end{figure}

Therefore, from now on we have that

$$n \geqslant 6 \textit{.} \eqno(1)$$

If all aforementioned properties (2-connectivity, minimum degree, lack of the required cycle)  remain, we add an edge to the graph. Obviously, 2-connectivity and minimum degree properties remain after adding an edge, so we add an edge if absence of the required cycle remains. Let us perform this procedure as long as we can, we stop the procedure if we cannot add any edge. We do not want to make new notations, so we use the same symbol $G$ for the graph we obtain after stopping the procedure. Consequently, $G$ has the following property:

$ $

For any $d_{1}, d_{2} \in V(G)$ such that $d_{1}d_{2} \notin E(G)$, there is a path between $d_{1}$ and $d_{2}$ such that the vertex set beyond this path is independent.

\begin{claim} \label{claimone} $G$ has a Hamiltonian path.

\end{claim}

%\textbf{Claim 1:} $G$ has a Hamiltonian path.

\begin{proof}

\begin{figure}[htb]
	\centering
	\includegraphics[width=0.6\columnwidth, keepaspectratio]{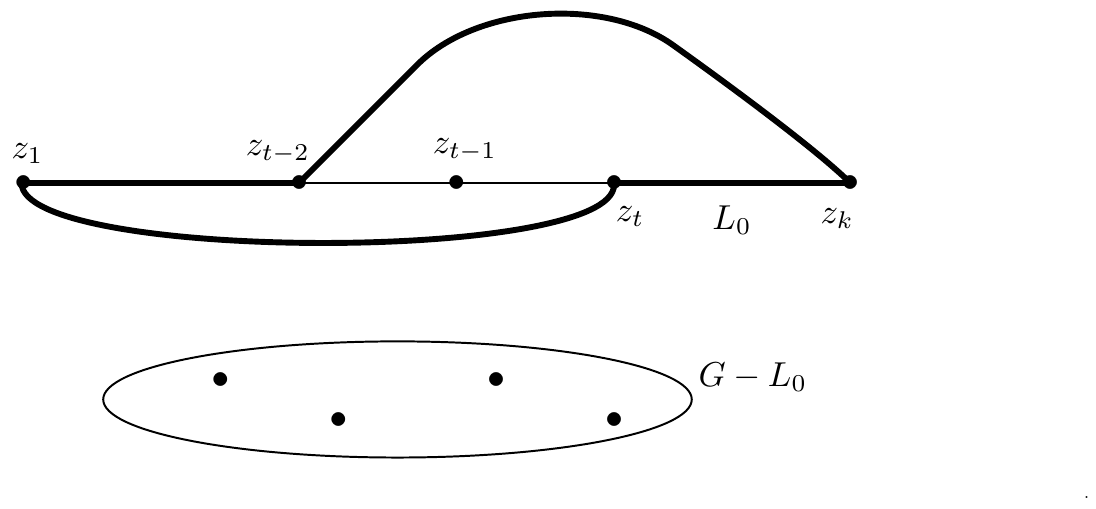}
	\caption{$z_{t} \in A$ and $z_{t - 2} \in B$ for some $t$.}
	\label{figuretwo}
\end{figure}

Consider the longest path $L_{0}$ in the graph $G$ such that its endpoints are not adjacent to each other and the vertex set beyond this path is independent. Assuming the converse, this is not a Hamiltonian path. Then there are $k < n$ vertices in $L_{0}$.  Let us number them in the order of passing the path: $z_{1}, z_{2}, ..., z_{k}$. We say that vertex $z_{i}$ is on the \textit{right} of vertex $z_{j}$ if $i > j$ and vertex $z_{i}$ is on the \textit{left} of vertex $z_{j}$ if $i < j$. Let $A = N_{G}(z_{1})$, $B = N_{G}(z_{k})$. It is clear that $A \cup B \subset L_{0}$ (otherwise, $L_{0}$ may be lengthened due to some vertex from $A \cup B \setminus L_{0}$). Then, since $d_{G}(z_{1}) \geqslant \frac{n}{3}$ and $d_{G}(z_{k}) \geqslant \frac{n}{3}$, $|A \cap L_{0}| \geqslant \frac{n}{3}$ and $|B \cap L_{0}| \geqslant \frac{n}{3}$.

Note that, for any $t$, the following property is satisfied: if $z_{t} \in A$ then $z_{t - 2} \notin B$. Assuming the converse, $z_{t} \in A$ and $z_{t - 2} \in B$ (see figure~\ref{figuretwo}). If $N_{G}(z_{t-1}) \subseteq L_{0}$ then note that there is a cycle of length $k - 1$ (see figure~\ref{figuretwo}). The vertex set beyond this cycle consists of the vertex set beyond $L_{0}$ (which is independent) and the vertex $z_{t - 1}$. Then, since $N_{G}(z_{t-1}) \subseteq L_{0}$, this cycle is suitable. Hence, $N_{G}(z_{t - 1}) \nsubseteq L_{0}$. Let $h \in N_{G}(z_{t - 1}) \cap (G - L_{0})$. Then there is a path of length $k + 1$ such that the vertex set beyond this path is independent: $z_{1}z_{2}...z_{t-2} z_{k}z_{k-1}z_{k-2}...z_{t}z_{t-1}h$ (see figure~\ref{figurethree}). This is a contradiction to the suggestion that $L_{0}$ is the longest path such that $V(G - L_{0})$ is independent. 

%The set of remaining vertices is independent insofar as this is the vertex set beyond $L_{0}$ (which is independent) without vertex $h$. 

\begin{figure}[htb]
	\centering
	\includegraphics[width=0.6\columnwidth, keepaspectratio]{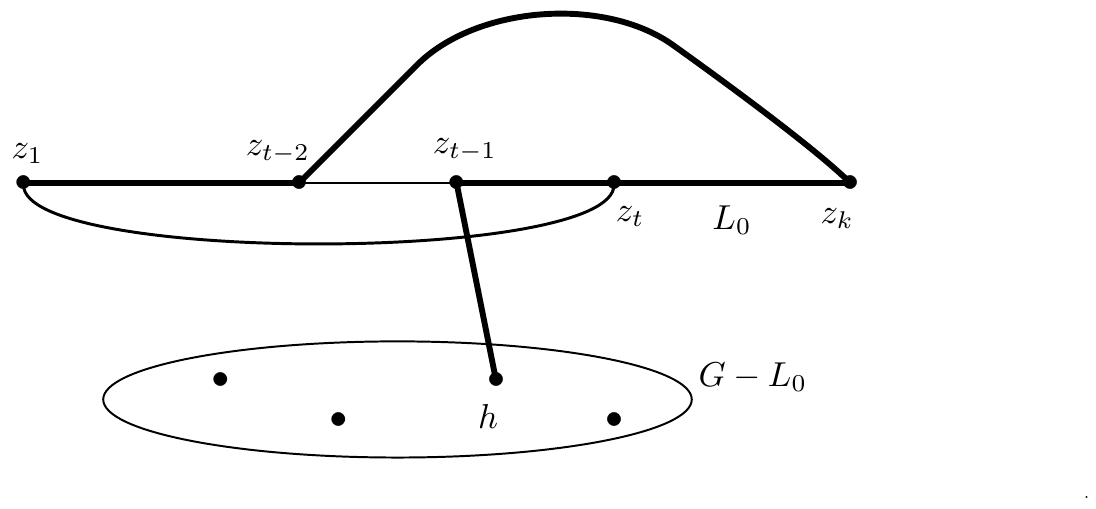}
	\caption{$z_{t} \in A$ and $z_{t - 2} \in B$ for some $t$, $N_{G}(z_{t - 1}) \cap (G - L_{0}) \neq \varnothing$.}
	\label{figurethree}
\end{figure}

Since the path $L_{0}$ is not Hamiltonian, $k < n$, and there exists a vertex $y \notin L_{0}$. Since $V(G - L_{0})$ is independent, $N_{G}(y) \subset V(L_{0})$. Let $\mu = |N_{G}(y)|$. Let $C$ be a set of all vertices $z_{t}$ such that $z_{t + 1}y \in E(G)$. Since $\delta (G) \geqslant \frac{n}{3}$, $\mu \geqslant \frac{n}{3}$ and $|C| \geqslant \mu$ (because $z_{1}$ cannot be a neighbour of $y$, otherwise the path $L_{0}$ can be obviously lengthened from $z_{1}...z_{k}$ to $yz_{1}...z_{k}$ with maintaining the condition about independence of the set of remaining vertices).

Note that $B \cap C = \varnothing$ (see figure~\ref{figurefour}).

\begin{figure}[htb]
	\centering
	\includegraphics[width=0.6\columnwidth, keepaspectratio]{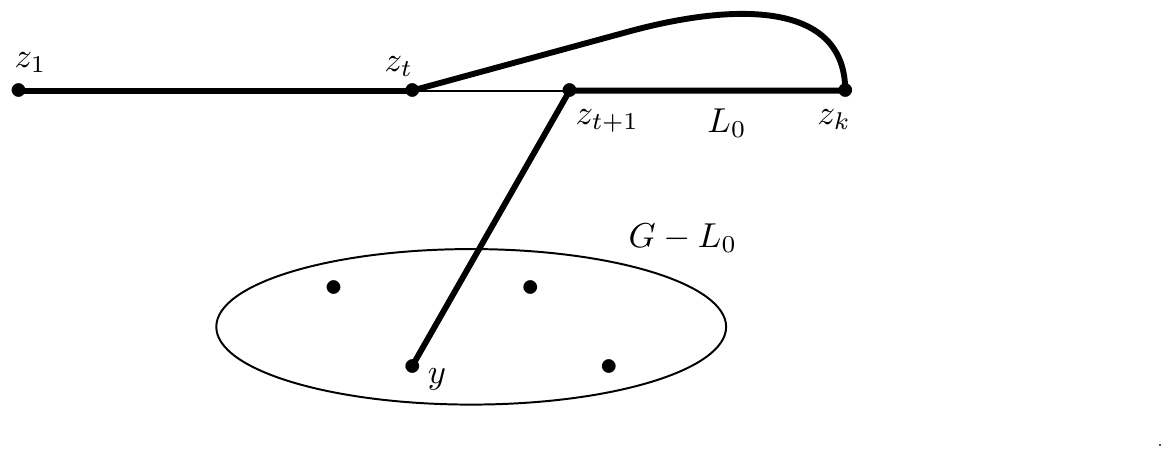}
	\caption{If $z_{t} \in B \cap C$ then $z_{t + 1}y, z_{t}z_{k} \in E(G)$. Then  the path (bold in the illustration) $z_{1}z_{2}...z_{t}z_{k}z_{k - 1}...z_{t + 1}y$ has length $k + 1$, and the vertex set beyond this path is independent, a contradiction to the choice of $L_{0}$.}
	\label{figurefour}
\end{figure}

Consequently, we proved the following:

1) $|A| \geqslant \frac{n}{3}$, $|B| \geqslant \frac{n}{3}$, $|C| \geqslant \frac{n}{3}$.

2) $B \cap C = \varnothing$.

3) If $z_{t} \in A$ then $z_{t - 2} \notin B$.

We say that $z_{t}$ is \textit{blocked} if $z_{t + 2} \in A$. It follows from the fact 3) that $\{\text{blocked vertices}\} \cap B = \varnothing$. Also, $\{\text{blocked vertices}\} \cap C = \varnothing$. Indeed, suppose that $z_{t}$ is blocked and $z_{t} \in C$. Then $z_{t + 1}y \in E(G)$ and $z_{1}z_{t + 2} \in E(G)$. Then the path $yz_{t + 1}z_{t}z_{t - 1}...z_{1}z_{t + 2}z_{t + 3}...z_{k}$ is longer than the path $L_{0}$, and the set of remaining vertices is independent (it is similar to what happens in figure~\ref{figurefour}). 

By definition, every vertex from $A$ (recall that $|A| \geqslant \frac{n}{3}$) \textit{generates} the blocked vertex except $z_{2}$ ($z_{1} \notin A$ because it is not neighbour to itself). Therefore, there are at least $\frac{n}{3} - 1$ blocked vertices. 

Thus, three sets \{blocked vertices\}, $B$, $C$ are pairwise disjoint and their cardinalities are at least $\frac{n}{3} - 1$, $\frac{n}{3}$, $\frac{n}{3}$ respectively. Since all these three sets are subsets of $\{z_{1}, ..., z_{k - 1}\}$ (obviously, $z_{k} \notin \{\text{blocked vertices}\} \cup B \cup C$), $k - 1 \geqslant \frac{n}{3} - 1 + \frac{n}{3} + \frac{n}{3} = n - 1$. But $k < n$ because the path $L_{0}$ is not Hamiltonian, a contradiction.

\end{proof}

By Claim~\ref{claimone}, $G$ has a Hamiltonian path. Let us number path's vertices in the order of passing the path: $x_{1}, x_{2}, ..., x_{n}$. As before, we say that $x_{i}$ is on the \textit{right} of $x_{j}$ if $i > j$ and $x_{i}$ is on the \textit{left} of $x_{j}$ if $i < j$. Since $d_G(x_{1}) \geqslant \frac{n}{3}$, there exists an index $i \geqslant \frac{n}{3} + 1$ such that $x_{1}x_{i} \in E(G)$. Vertices $x_{1}$, $x_{2}$, ..., $x_{i}$ form a cycle (in this order of passing) and all the remaining vertices $x_{i + 1}$, $x_{i + 2}$, ..., $x_{n}$ form a path (in this order of passing) (see figure~\ref{figurefive}). Hence, the graph $G$ contains a cycle of length at least $\frac{n}{3} + 1$ such that all the remaining vertices form a path.

\begin{figure}[htb]
	\centering
	\includegraphics[width=0.6\columnwidth, keepaspectratio]{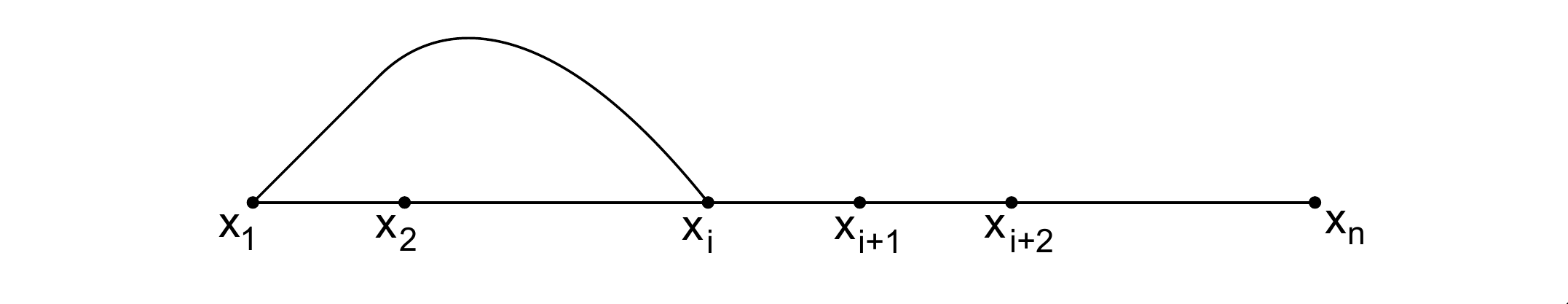}
	\caption{The graph $G$ has a cycle $x_{1}x_{2}...x_{i}$, all the remaining vertices form a path $x_{i + 1}x_{i + 2}...x_{n}$.}
	\label{figurefive}
\end{figure}

\begin{figure}[htb]
	\centering
	\includegraphics[width=0.6\columnwidth, keepaspectratio]{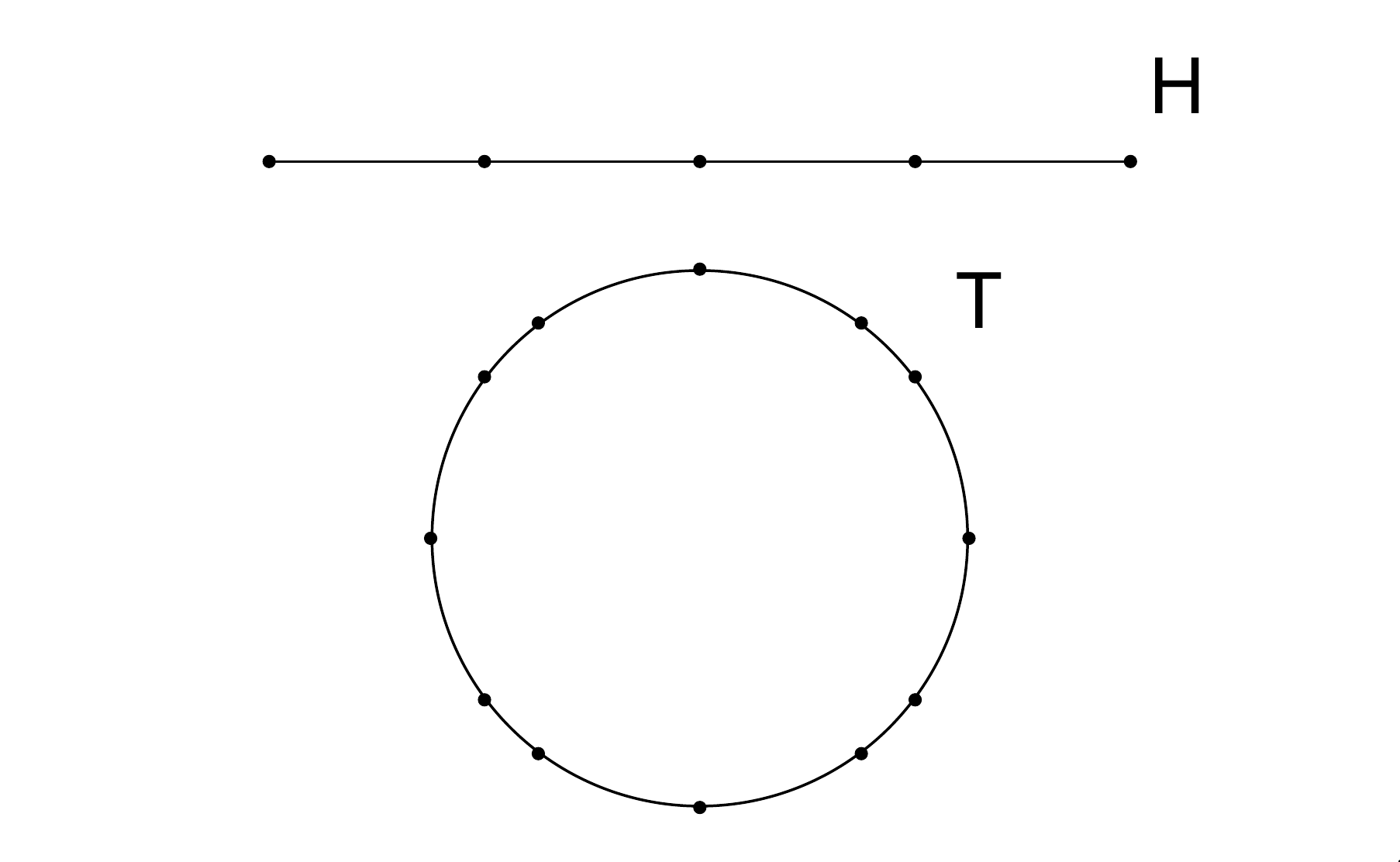}
	\caption{Vertices of the graph $G$ are divided into 2 disjoint sets: the cycle $T$ and the path $H$.}
	\label{figuresix}
\end{figure}

Let $r$ be the largest integer such that there exists a cycle of length $r$ (let this cycle be $T$) such that all the remaining vertices form a path (let this path be $H$). We proved that $r \geqslant \frac{n}{3} + 1$. The cycle $T$ has $r$ vertices, therefore the path $H$ has $n - r$ vertices and $n - r - 1$ edges (see figure~\ref{figuresix}).

\begin{claim} \label{claimtwo} $r \geqslant \frac{n}{2}$.

\end{claim}

\begin{proof}

%$ $

%\textbf{Claim 2:} $r \geqslant \frac{n}{2}$. 

%$ $

Assuming the converse, $r < \frac{n}{2}$. Let $N_{G}(x_{1}) = \{x_{\alpha_{1}}, ..., x_{\alpha_{f}}\}$ and $N_{G}(x_{n}) = \{x_{\beta_{1}}, ..., x_{\beta_{g}}\}$. Since $\delta(G) \geqslant \frac{n}{3}$, $f \geqslant \frac{n}{3}$ and $g \geqslant \frac{n}{3}$. Note that, for any $1 \leqslant i \leqslant f$ and $1 \leqslant j \leqslant g$, property $\alpha_{i} - 1 \neq \beta_{j} + 1$ holds because otherwise there is a cycle of length $n - 1$ (it contains all the vertices except $x_{\alpha_{i} - 1} = x_{\beta_{j} + 1}$, that is: $x_{1}x_{2}...x_{\beta_{j}}x_{n}...x_{\alpha_{i}}$). This cycle is suitable because the vertex set beyond this cycle (that is, one vertex) is independent. Also note that, for any $1 \leqslant i \leqslant f$ and $1 \leqslant j \leqslant g$, inequalities $\alpha_{i} < \frac{n}{2}$ and $\beta_{j} > \frac{n}{2}$ hold (otherwise there is a suitable cycle of length at least $\frac{n}{2}$ such that all the remaining vertices form a path: $x_{1}x_{2}...x_{\alpha_{i}}$ or $x_{\beta_{j}}x_{\beta_{j} + 1}...x_{n}$). Consider the vertex $x_{\lceil \frac{n}{2}\rceil}$. We proved that it is on the right of all the neighbours of $x_{1}$ and it is on the left of all the neighbours of $x_{n}$. Note that, for any $1 \leqslant i \leqslant f$ and $1 \leqslant j \leqslant g$, properties $x_{\lceil \frac{n}{2}\rceil}x_{\alpha_{i} - 1} \notin E(G)$ and $x_{\lceil \frac{n}{2}\rceil}x_{\beta_{i} + 1} \notin E(G)$  hold (otherwise there is a cycle of length at least $\frac{n}{2}$ such that all the remaining vertices form a path: see figures~\ref{figureseven} a and ~\ref{figureseven} b respectively). Since all the vertices $x_{\alpha_{i} - 1}$ and $x_{\beta_{i} + 1}$ are distinct,  $x_{\lceil \frac{n}{2}\rceil}$ is not adjacent to at least  $\frac{2n}{3}$ distinct vertices and it is not adjacent to itself. Then, since $v(G) = n$, $d_{G}(x_{\lceil \frac{n}{2}\rceil}) \leqslant \frac{n}{3} - 1$, contrary to $\delta(G) \geqslant \frac{n}{3}$. 

\end{proof}

\begin{figure}[htb]
	\centering
	\includegraphics[width=1.2\columnwidth, keepaspectratio]{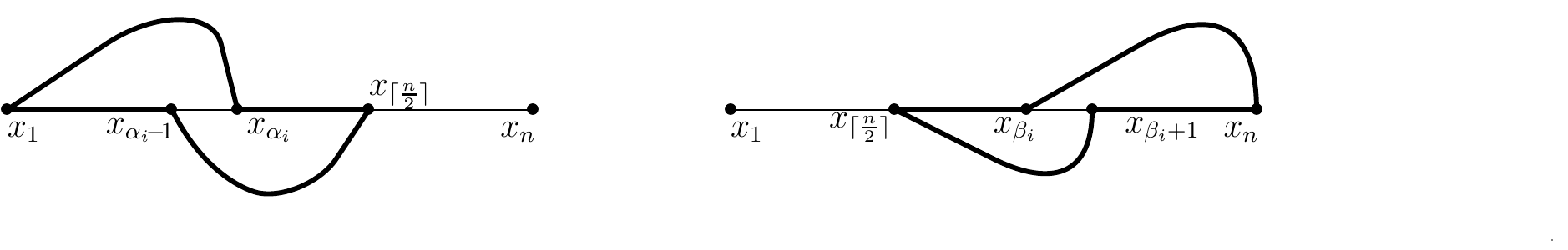}
	\caption{If $x_{\lceil \frac{n}{2}\rceil}x_{\alpha_{i} - 1} \in E(G)$ or $x_{\lceil \frac{n}{2}\rceil}x_{\beta_{i} + 1} \in E(G)$ then there is a cycle (bold in the figure) of length at least $\frac{n}{2}$ such that all the remaining vertices form a path.}
	\label{figureseven}
\end{figure}

Let us come back directly to the proof of Theorem~\ref{mytheorem} and consider 2 cases. 

$ $

\textbf{Case 1.} \textit{There are 2 distinct vertices of the cycle $T$ such that the first vertex is adjacent to some endpoint (let this endpoint be $a$) of the path $H$ and the second vertex is adjacent to the second endpoint (let this endpoint of the path $H$ be $b$).}

$ $

Let $\rho_{a} = |N_{G}(a) \cap V(T) \setminus N_{G}(b)|$, let vertices from $N_{G}(a) \cap V(T) \setminus N_{G}(b)$ be \textit{a-vertices} or \textit{vertices of type a}. Let $\rho_{b} = |N_{G}(b) \cap V(T) \setminus N_{G}(a)|$, let vertices from $N_{G}(b) \cap V(T) \setminus N_{G}(a)$ be \textit{b-vertices} or \textit{vertices of type b}. Also, let $\rho_{ab} = |N_{G}(a) \cap N_{G}(b) \cap V(T)|$, let vertices from $N_{G}(a) \cap N_{G}(b) \cap V(T)$ be \textit{ab-vertices} or \textit{vertices of type ab}). Note that every vertex of the cycle $T$ cannot pertain to 2 types simultaneously but it can pertain to none of the types (if it is not adjacent to $a$ and $b$). Note that two vertices of the same type cannot be neighbours in the order of passing the cycle (otherwise the cycle can be lengthened by one of the vertices $a$, $b$, and the remaining vertices form a path; or a Hamiltonian cycle exists, see figure~\ref{figureeight}).

\begin{figure}[htb]
	\centering
	\includegraphics[width=0.7\columnwidth, keepaspectratio]{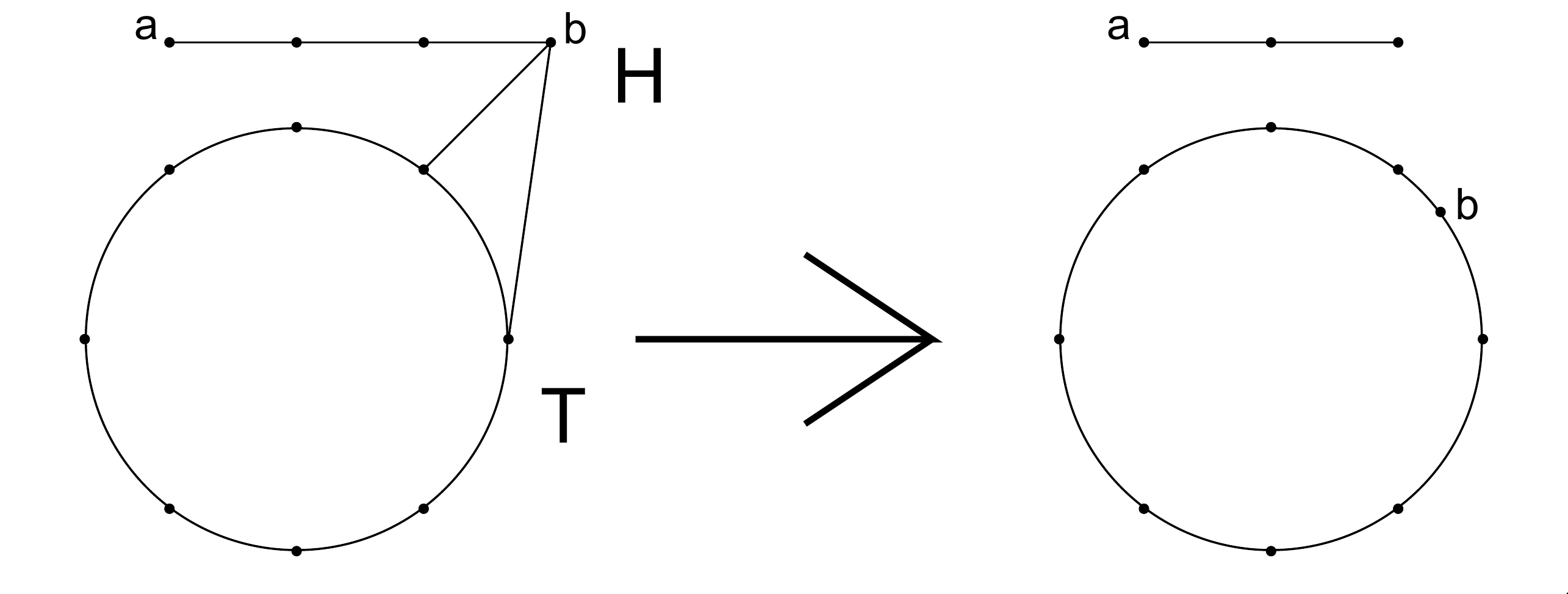}
	\caption{If an endpoint of $H$ ($b$ in this illustration) is adjacent to two neighbours in the order of passing the cycle then the cycle $T$ may be lengthened due to this endpoint and the remaining vertices form a path, a contradiction to the choice of the cycle $T$.}
	\label{figureeight}
\end{figure}

Note that there are no 2 distinct vertices of the cycle $T$ with three conditions:

1) first vertex is adjacent to $a$;

2) second vertex is adjacent to $b$;

3) the distance between these vertices by edges of the cycle $T$ is at most $n - r$.

Assume the converse. Note that there is a cycle (it passes the edges of the major half of the cycle $T$ between these vertices, endpoints of this section are adjacent to $a$ and $b$, and it passes the path $H$, see figure~\ref{figurenine}) such that the vertex set beyond this cycle forms a path and its cardinality is at most $n - r - 1$. Hence, there are at least $r + 1$ vertices in this cycle, contrary to the definition of $r$: we have found a cycle that satisfies all the requirements of the cycle in the definition of $r$, but this cycle is larger than the cycle $T$.

\begin{figure}[htb]
	\centering
	\includegraphics[width=0.35\columnwidth, keepaspectratio]{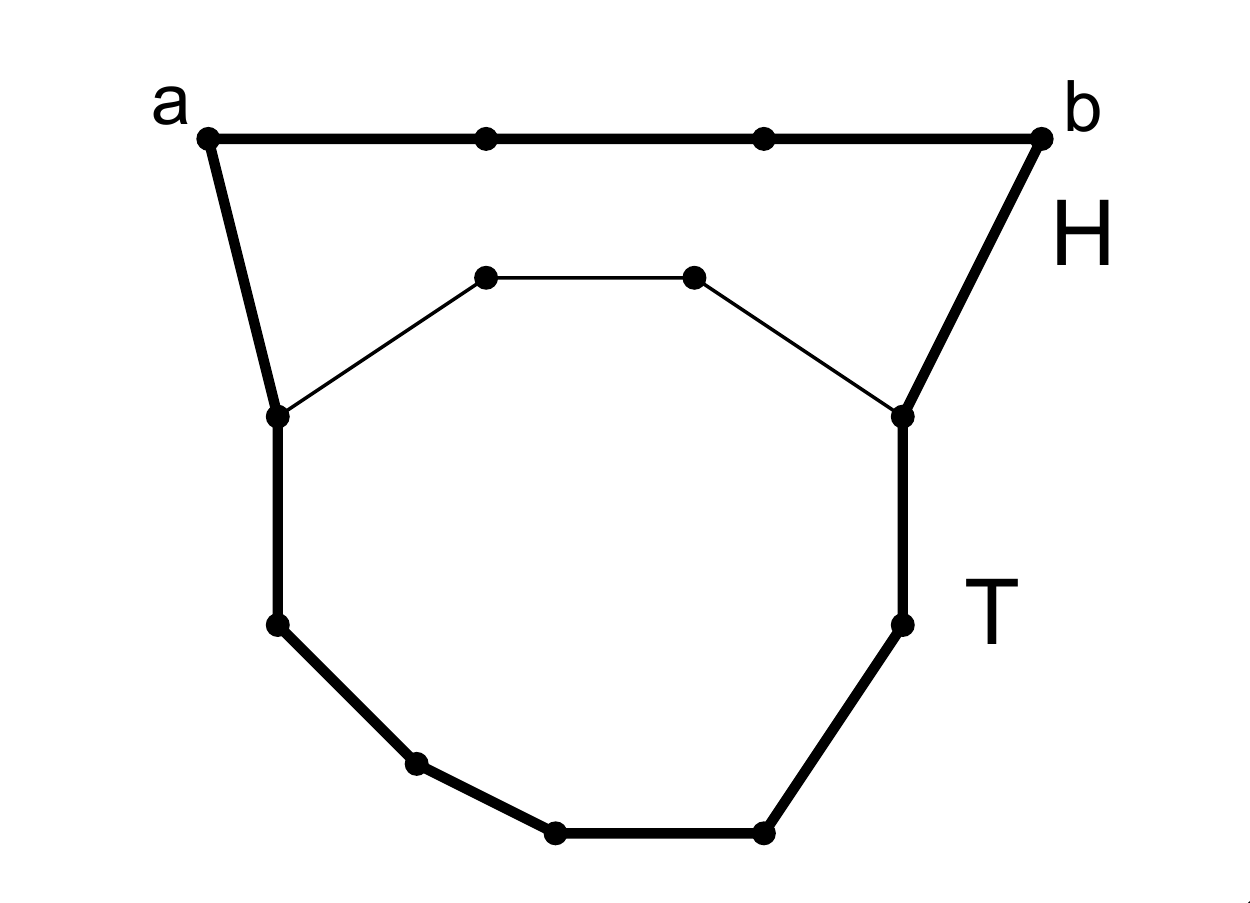}
	\caption{If one vertex of the cycle $T$ is adjacent to $a$, the other one is adjacent to $b$ and the distance between them by edges of the cycle $T$ is at most $n - r$ then there exists another cycle (bold in this figure), and its existence is contrary to the definition of $r$.}
	\label{figurenine}
\end{figure}

Therefore, if there are 2 distinct vertices of different types (or both of them have type $ab$) then the distance between them is at least $n - r + 1$. 

%Let us have 2 sections.

$ $

\textbf{Case 1.1.} \textit{The cycle T has neither a-vertices or b-vertices.}

$ $

By the condition of the case 1.1, $\rho_{a} = \rho_{b} = 0$. Then, by the condition of the case 1, $\rho_{ab} \geqslant 2$.

The vertex $a$ is adjacent to $\rho_{ab}$ vertices of the cycle $T$ and it may be adjacent to some vertices of the path $H$, not including itself. Consequently, the vertex $a$ is adjacent to at most $\rho_{ab} + n - r - 1$ vertices. Since $\delta(G) \geqslant \frac{n + 2}{3}$,  $\rho_{ab} + n - r - 1 \geqslant \frac{n + 2}{3} \Rightarrow \frac{2n - 5}{3} + \rho_{ab} \geqslant r$.

Since the distance between every $ab-$vertices by the edges of the cycle $T$ is at least $n - r + 1$, the cycle $T$ contains at least $\rho_{ab}(n - r + 1)$ edges. Then, since the cycle $T$ has $r$ edges, $r \geqslant \rho_{ab}(n - r + 1) \Rightarrow (\rho_{ab} + 1)r \geqslant \rho_{ab}(n + 1)$. By applying the estimate of $r$ from the previous paragraph, we obtain that 

$$(\rho_{ab} + 1)(\frac{2n - 5}{3} + \rho_{ab}) \geqslant \rho_{ab}(n + 1) \Rightarrow \frac{2n - 5}{3}\rho_{ab} + \frac{2n - 5}{3} + \rho_{ab}^{2} \geqslant n\rho_{ab} \Rightarrow $$

$$\Rightarrow 3\rho_{ab}^{2} - 5\rho_{ab} - 5 \geqslant (\rho_{ab} - 2)n \textit{.} $$

If $\rho_{ab} = 2$ then this inequality does not hold. Therefore, $\rho_{ab} \geqslant 3 \Rightarrow \rho_{ab} - 2 > 0$. Hence,

$$3\rho_{ab}^2 - 5\rho_{ab} - 5 \geqslant (\rho_{ab} - 2)n \Rightarrow \frac{3\rho_{ab}^2 - 5\rho_{ab} - 5}{\rho_{ab} - 2} \geqslant n \textit{.}$$

It is clear that  $$3\rho_{ab} + 1 > \frac{3\rho_{ab}^2 - 5\rho_{ab} - 5}{\rho_{ab} - 2}\textit{.} $$

Consequently,

$$3\rho_{ab} + 1 > n \Rightarrow \rho_{ab} > \frac{n - 1}{3}\textit{.} $$

$ $

Recall that $r \geqslant \rho_{ab}(n - r + 1)$. Note that $n - r + 1 \geqslant 0$ because otherwise $r = n$, and the cycle $T$ is Hamiltonian and therefore suitable. Then, since $\rho_{ab} > \frac{n - 1}{3}$, we get that $$r > \frac{n - 1}{3}(n - r + 1)\textit{.}$$

It is clear that $r \leqslant n$. Then $n - r + 1 > 0$. If $n - r + 1 = 1$ then $r = n$, and hence the cycle $T$ is Hamiltonian and therefore suitable. If $n - r + 1 = 2$ then $r = n - 1$, and hence the cycle $T$ does not contain only one vertex, so it is suitable too. If $n - r + 1 \geqslant 3$ then $r > \frac{n - 1}{3}(n - r + 1) \Rightarrow r > n - 1$. Thus, this case is handled.

$ $

\textbf{Case 1.2.} \textit{The cycle $T$ contains an a-vertex or a b-vertex.}

$ $

A \textit{section of type a} or an \textit{a-section} is a consecutive (in the order of passing the cycle $T$) vertex set such that, firstly, it consists of all the vertices between two a-vertices (hence, extreme vertices of an a-section are a-vertices), secondly, it cannot be lengthened, thirdly, it does not include vertices of other types (a-section can consist of one vertex). \textit{Sections of type b} or \textit{b-sections}, \textit{sections of type ab} or \textit{ab-sections} have analogous definitions. Edges of the cycle are naturally divided into at most 4 groups: edges in a-sections, edges in b-sections, edges in ab-sections and edges between 2 distinct sections. Let $\gamma_{a}$ be the number of $a-$sections, let $\gamma_{b}$ be the number of $b-$sections, let $\gamma_{ab}$ be the number of $ab-$sections. By the conditions of the cases 1 and 1.1, there are at least 2 positive integers among three numbers $\rho_{a}, \rho_{b}, \rho_{ab}$ (respectively, there are at least 2 positive integers among three numbers $\gamma_{a}, \gamma_{b}, \gamma_{ab}$). Hence, $\gamma_{a} + \gamma_{b} + \gamma_{ab} \geq 2$. Consequently, the cycle $T$ does not consist of only one section.

It follows from the observation before the case 1.1 that there are at least $n - r + 1$ edges (in the order of passing the cycle $T$) between 2 distinct sections. Also, there are at least $n - r + 1$ edges (in the order of passing the cycle $T$) between 2 distinct ab-vertices in an ab-section. Let us consider all the edges of the cycle $T$ that does not constitute edges from a-sections and b-sections. For this purpose, we fix some direction of the cycle $T$. Walking around the cycle $T$ in this direction, there are at least $n - r + 1$ edges between distinct sections after passing an a-section or a b-section. Also, there are at least $n - r + 1$ edges between distinct sections or in ab-section after an ab-vertex. Thus, there are at least $(n - r + 1)(\rho_{ab} + \gamma_{a} + \gamma_{b})$ edges from the cycle $T$ that do not constitute edges inside a-section or b-section.

Let us consider any a-section. Let it be $s$ a-vertices in this section. It is clear that if $s = 1$ then this section consists of at least $2s - 2$ edges. Since an a-vertex does not have a-vertices and b-vertices among its 2 neighbours in the order of passing the cycle (figure~\ref{figureeight} is dedicated to that), if $s \geqslant 2$ then this section has at least $2s - 1$ edges. Hence, any $s$ a-section has at least $2s - 2$ edges. Therefore, there are at least $2\rho_{a} - 2\gamma_{a}$ edges inside a-sections. Similarly, there are at least $2\rho_{b} - 2\gamma_{b}$ edges inside b-sections. Hence, there are at least $2(\rho_{a} + \rho_{b}) - 2(\gamma_{a} + \gamma_{b})$ edges inside a-sections and b-sections.  

Thus, there are at least  $(n - r + 1)(\gamma_{a} + \gamma_{b} + \rho_{ab}) + 2(\rho_{a} + \rho_{b}) - 2(\gamma_{a} + \gamma_{b})$ edges in the cycle $T$. Since the cycle $T$ consists of $r$ edges, the following inequality holds:

$$r \geqslant (n - r + 1)(\gamma_{a} + \gamma_{b} + \rho_{ab}) + 2(\rho_{a} + \rho_{b}) - 2(\gamma_{a} + \gamma_{b}) \geqslant $$

$$ \geqslant (n - r - 2)(\gamma_{a} + \gamma_{b} + \rho_{ab}) + 2(\rho_{a} + \rho_{b}) + 3\rho_{ab}\textit{.}$$

It follows from $\gamma_{a} + \gamma_{b} + \gamma_{ab} \geqslant 2$ and $\rho_{ab} \geqslant \gamma_{ab}$ that $\gamma_{a} + \gamma_{b} + \rho_{ab} \geqslant 2$. Also, $n - r - 2 \geqslant 0$, because otherwise $r = n$ or $r = n - 1$, these cases are handled before the case 1.2. Consequently,

$$r \geqslant 2(n - r - 2) + 2(\rho_{a} + \rho_{b}) + 3\rho_{ab} \Rightarrow 3r \geqslant 2(n - 2) + 2(\rho_{a} + \rho_{b}) + 3\rho_{ab} \Rightarrow $$

$$\Rightarrow r \geqslant \frac{2}{3}(\rho_{a} + \rho_{b}) + \frac{2}{3}n - \frac{4}{3} + \rho_{ab}\textit{.}$$

The vertex $a$ has $\rho_{a} + \rho_{ab}$ neighbours among vertices of the cycle $T$. Also, it can be adjacent to all the vertices of the path $H$ except itself. Therefore, the vertex $a$ is adjacent to at most $n - r - 1 + \rho_{a} + \rho_{ab}$ vertices. On the other hand, $d_{G}(a) \geqslant \frac{n + 2}{3}$. Therefore, $\frac{n + 2}{3} \leqslant n - r - 1 + \rho_{a} + \rho_{ab}$. The same is true for the vertex $b$. Hence, $\frac{n + 2}{3} \leqslant n - r - 1 + \rho_{b} + \rho_{ab}$. Consequently, $\frac{n + 2}{3} \leqslant n - r - 1 + min(\rho_{a}, \rho_{b}) + \rho_{ab}$. Thus,

$$ r \leqslant \frac{2n - 5}{3} + min(\rho_{a}, \rho_{b}) + \rho_{ab}\textit{.}$$

It follows from this and the previous inequalities that

$$\frac{2n - 5}{3} + min(\rho_{a}, \rho_{b}) + \rho_{ab} \geqslant \frac{2}{3}(\rho_{a} + \rho_{b}) + \frac{2}{3}n - \frac{4}{3} + \rho_{ab} \Rightarrow $$

$$\Rightarrow min(\rho_{a}, \rho_{b}) \geqslant \frac{2}{3}(\rho_{a} + \rho_{b}) + \frac{1}{3} \textit{.}$$

Then it follows from $\frac{2}{3}(\rho_{a} + \rho_{b}) \geqslant \frac{4}{3}min(\rho_{a}, \rho_{b})$ that

$$ min(\rho_{a}, \rho_{b}) \geqslant \frac{4}{3}min(\rho_{a}, \rho_{b}) + \frac{1}{3} \Rightarrow 0 \geqslant \frac{1}{3}min(\rho_{a}, \rho_{b}) + \frac{1}{3}\textit{.}$$

This is impossible, a contradiction.

$ $

\textbf{Case 2.} \textit{There are no 2 distinct vertices of the cycle $T$ such that one of them is adjacent to $a$ and the second vertex is adjacent to $b$.}

$ $

In this case, one of the vertices $a$, $b$ has at most 1 neighbour among vertices of the cycle $T$. Without loss of generality, $a$ has at most 1 neighbour among vertices of the cycle $T$. Since $d_{G}(a) \geqslant \frac{n}{3}$, $a$ has at least $\frac{n}{3} - 1$ neighbours among vertices of the path $H$. Hence, the path $H$ contains (including $a$) at least $\frac{n}{3} - 1 + 1 = \frac{n}{3}$ vertices. Then, since the path $H$ has $n - r$ vertices, $$n - r \geqslant \frac{n}{3} \Rightarrow r \leqslant \frac{2n}{3}\textit{.} \eqno(2)$$

\begin{claim} \label{claimthree} a) Any endpoint of $H$ has at most $\frac{r}{3}$ neighbours among vertices of the cycle $T$. 

b) If the graph $G(H)$ is Hamiltonian then any vertex of the graph $G(H)$ has at most $\frac{r}{3}$ neighbours among vertices of the cycle $T$. 

\end{claim}

%\textbf{Claim 3:} a) Any endpoint of $H$ has at most $\frac{r}{3}$ neighbours amongst vertices of the cycle $T$. 

%b) If the graph $G(H)$ is Hamiltonian then any vertex of the graph $G(H)$ has at most $\frac{r}{3}$ neighbours amongst vertices of the cycle $T$. 

\begin{proof}

a) Firstly, let us prove that for the vertex $a$. Assuming the converse, $a$ has more than $\frac{r}{3}$ neighbours among vertices of the cycle $T$. Since $a$ has at most one neighbour among vertices of the cycle $T$, $1 > \frac{r}{3} \Rightarrow r < 3$. By Claim~\ref{claimtwo}, $r \geqslant \frac{n}{2}$. Hence, $3 > \frac{n}{2} \Rightarrow n < 6$, contrary to (1).

\begin{figure}[htb]
	\centering
	\includegraphics[width=0.4\columnwidth, keepaspectratio]{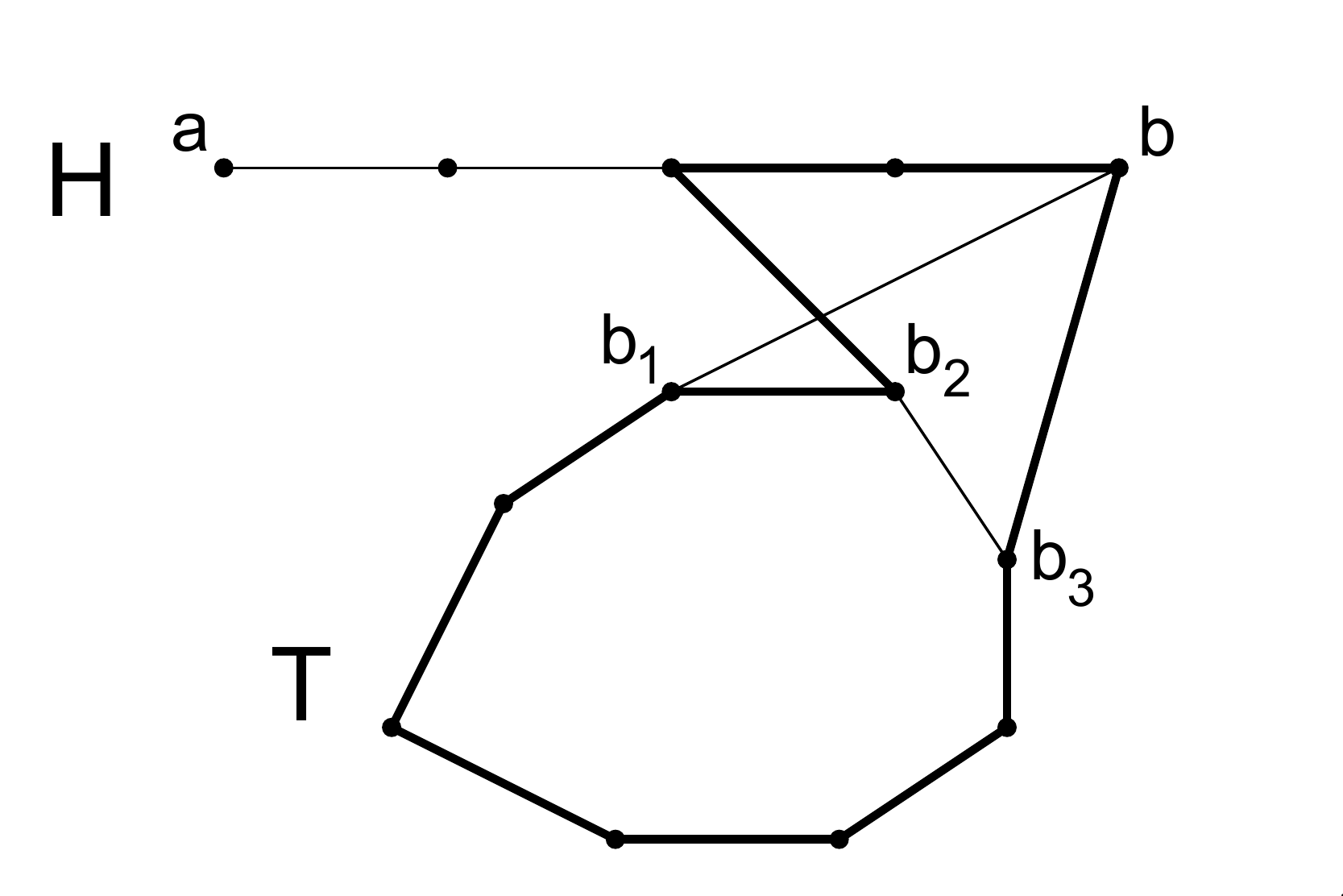}
	\caption{$b_{2}$ cannot be adjacent to some vertex of the path $H$, because otherwise there is a cycle (bold in the illustration), and its existence is contrary to the choice of the cycle $T$.}
	\label{figureten}
\end{figure}

It remains to prove that for the vertex $b$. Let us consider $k_{b}$ neighbours of the vertex $b$ in the cycle $T$ (keeping terminology, they are \textit{vertices of the type $b$} or \textit{b-vertices}). Note that b-vertices cannot be neighbours in the order of passing the cycle, because otherwise the cycle $T$ may be lengthened due to the vertex $b$, and the remaining vertices form a path (this is illustrated in figure~\ref{figureeight}). 

Assume that there are two b-vertices (we denote them by $b_{1}$ and $b_{3}$) such that they have exactly one vertex (we denote it by $b_{2}$) between them in the order of passing the cycle $T$. We proved that $bb_{2} \notin E(G)$. Note that $e_{G}(b_{2}, H) = 0$, because otherwise there is a cycle that is larger than $T$, and the remaining vertices form a path (see figure~\ref{figureten}). Hence, $N_{G}(b_{2}) \subset T$. Note that if two neighbours of $b_{2}$ are adjacent by the edge of the cycle $T$ then $b_{2}$ may be 'inserted' between these neighbours, and hence there is a cycle that includes all the vertices of the cycle $T$ and the vertex $b$ (hence, it is larger than $T$), and the vertex set beyond this cycle forms a path (see figure~\ref{figureeleven}). Hence, there are no two neighbours of $b_{2}$ that are neighbours in the order of passing the cycle $T$. Since the cycle $T$ consists of $r - 1$ vertices except $b_{2}$, $b_{2}$ is adjacent to at most $\lceil \frac{r - 1}{2}\rceil$ vertices. Thus, $\lceil \frac{r - 1}{2}\rceil \geqslant \frac{n + 2}{3} \Rightarrow \frac{r}{2} \geqslant \frac{n + 2}{3} \Rightarrow r \geqslant \frac{2(n + 2)}{3}$ because $d_{G}(b_{2}) \geqslant \frac{n + 2}{3}$. This is a contradiction to (2). 

\begin{figure}[htb]
	\centering
	\includegraphics[width=0.45\columnwidth, keepaspectratio]{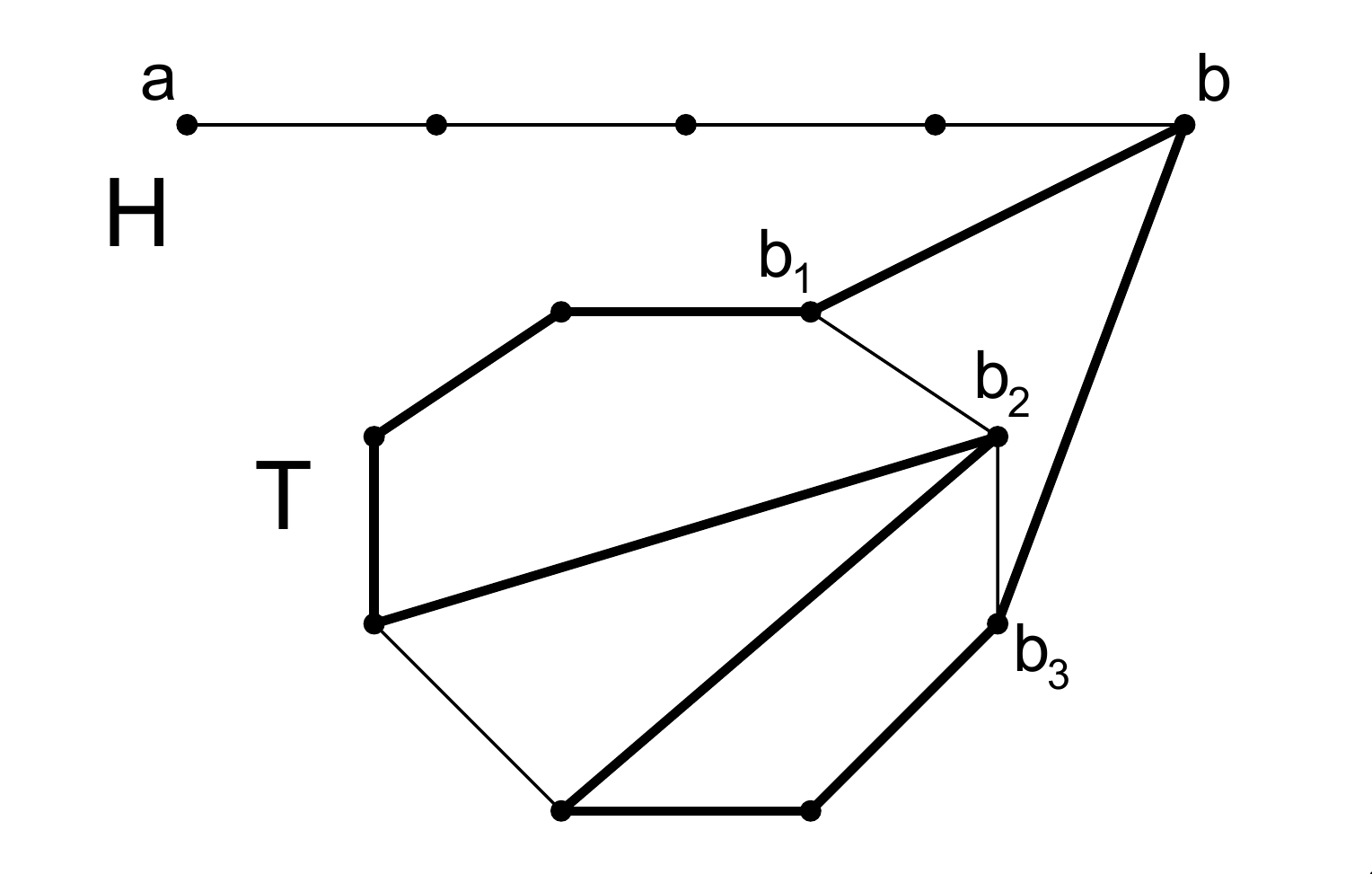}
	\caption{If $b_{2}$ is adjacent to two vertices of the cycle $T$ that are adjacent to each other by an edge of the cycle $T$ then there is a cycle (bold in this figure), and its existence is contrary to the choice of the cycle $T$.}
	\label{figureeleven}
\end{figure}

Therefore, there are no two b-vertices in the cycle $T$ that are adjacent to each other by an edge of the cycle $T$ or have exactly one vertex between them in the order of passing the cycle $T$. Consequently, $b$ has at most $\frac{r}{3}$ neighbours among vertices of the cycle $T$ because $T$ has $r$ vertices.

\vspace{3.5mm}

b) If the graph $G(H)$ is Hamiltonian then any vertex of the path $H$ may be an endpoint of some Hamiltonian path in the graph $G(V(H))$. Applying Claim~\ref{claimthree} (item a) to this vertex and this Hamiltonian path, we obtain the desired property. 

\end{proof}

\begin{figure}[htb]
	\centering
	\includegraphics[width=0.35\columnwidth, keepaspectratio]{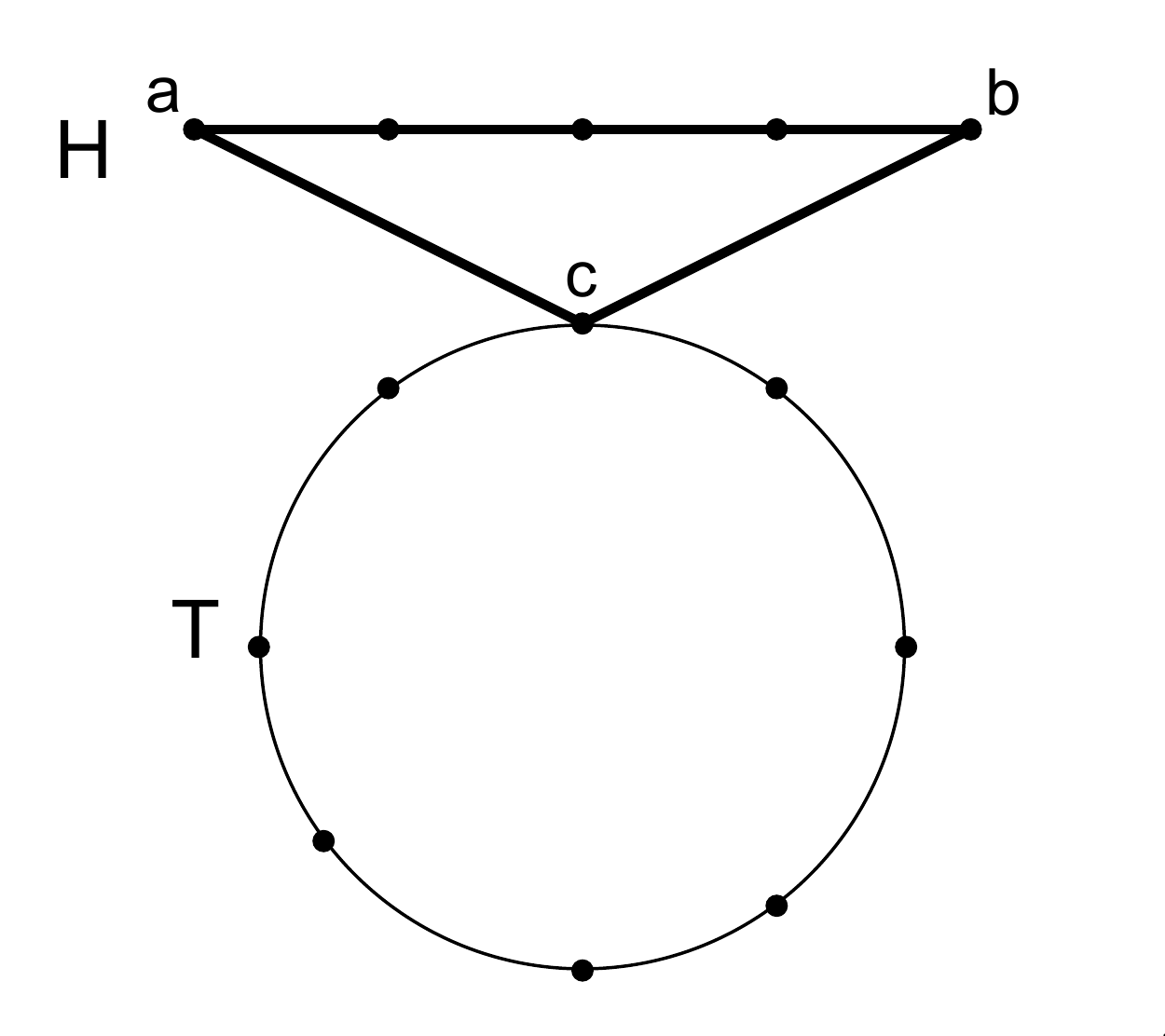}
	\caption{The path $H$ and the vertex $c$ form a cycle (bold in this illustration) such that the vertex set beyond this cycle forms a path.}
	\label{figuretwelve}
\end{figure}

\begin{claim} \label{claimfour} $G(V(H))$ has a Hamiltonian cycle.

\end{claim}

%$ $

%\textbf{Claim 4:} Vertices of the path $H$ form a cycle. 

%$ $

\begin{proof} Let us consider 2 cases.

$ $

\textbf{Case a:} \textit{$e_{G}(a, T) \geqslant 1$ and $e_{G}(b, T) \geqslant 1$.}

$ $

By the condition of the case 2, $e_{G}(a, T) = 1$, $e_{G}(b, T) = 1$ and $|N_{G}(\{a, b\}) \cap T| = 1$. Let $\{c\} = N_{G}(\{a, b\}) \cap T$, $ca, cb \in E(G)$. Consider subgraph $G(H)$. Since $a$ and $b$ have one neighbour in the cycle $T$, $d_{G(H)}(a) \geqslant \frac{n}{3} - 1$, $d_{G(H)}(b) \geqslant \frac{n}{3} - 1$. The graph $G(H)$ has a Hamiltonian path $H$. Applying Lemma~\ref{lemma} to the path $H$ that has $n - r$ vertices: if $d_{G(H)}(a) + d_{G(H)}(b) \geqslant n - r$ then graph $G(H)$ has a Hamiltonian cycle (this is the desired property). Otherwise,  $n - r > d_{G(H)}(a) + d_{G(H)}(b) \geqslant \frac{n}{3} - 1 + \frac{n}{3} - 1 = \frac{2n}{3} - 2 \Rightarrow r < \frac{n}{3} + 2$. Note that the path $H$ and the vertex $c$ form a cycle on $n - r + 1$ vertices such that the vertex set beyond this cycle forms a path (see figure~\ref{figuretwelve}). Then, by the definition of $r$, the following holds: $r \geqslant n - r + 1 \Rightarrow r \geqslant \frac{n + 1}{2}$. If $r > \frac{n + 1}{2}$ then $r \geqslant \frac{n}{2} + 1$. Then it follows from $r < \frac{n}{3} + 2$ that $\frac{n}{3} + 2 > \frac{n}{2} + 1 \Rightarrow n < 6$, a contradiction to (1). Consequently, $r = \frac{n+1}{2}$. Then $n$ is odd, and since $r < \frac{n}{3} + 2$, we have that $\frac{n}{3} + 2 > \frac{n + 1}{2} \Rightarrow n < 9$. Since $n$ is odd and (1) holds, $n = 7$, $r = \frac{n + 1}{2} = 4$. It follows from the previous observations that the graph has two cycles (let them be $T_{1}$ and $T_{2}$) on 4 vertices that have common vertex $c$. And there is no such a cycle on 5 vertices that the remaining 2 vertices are adjacent to each other (moreover, there is no such a cycle on 5 vertices that the remaining 2 vertices are not adjacent, because otherwise this cycle meets all the requirements). Therefore, $G$ does not have a cycle on 5 vertices. Since $G$ is 2-connected, $G - \{c\}$ is connected. Hence, there is an edge $c_{1}c_{2} \in E(G)$ for some $c_{1} \in V(T_{1}) \setminus \{c\}$ and $c_{2} \in V(T_{2}) \setminus \{c\}$ (see figure~\ref{figurethirteen}). Since $T_{1}$ and $T_{2}$ are cycles and $|T_{1}| = |T_{2}| = 4$, there exist paths from $c_{1}$ and $c_{2}$ to $c$ whose lengths are at least 2 and these paths pass the edges of the cycles $T_{1}$ and $T_{2}$ respectively. Concatenation of these paths and the edge $c_{1}c_{2}$ form a cycle whose size is at least 5. Since the graph does not have a cycle of size 5, this cycle's size is 6 or 7. Hence, it is Hamiltonian or this cycle contains all the vertices except one and hence it is suitable.

$ $

\textbf{Case b:} \textit{$e_{G}(a, T) = 0$ or $e_{G}(b, T) = 0$.}

\begin{figure}[htb]
	\centering
	\includegraphics[width=0.45\columnwidth, keepaspectratio]{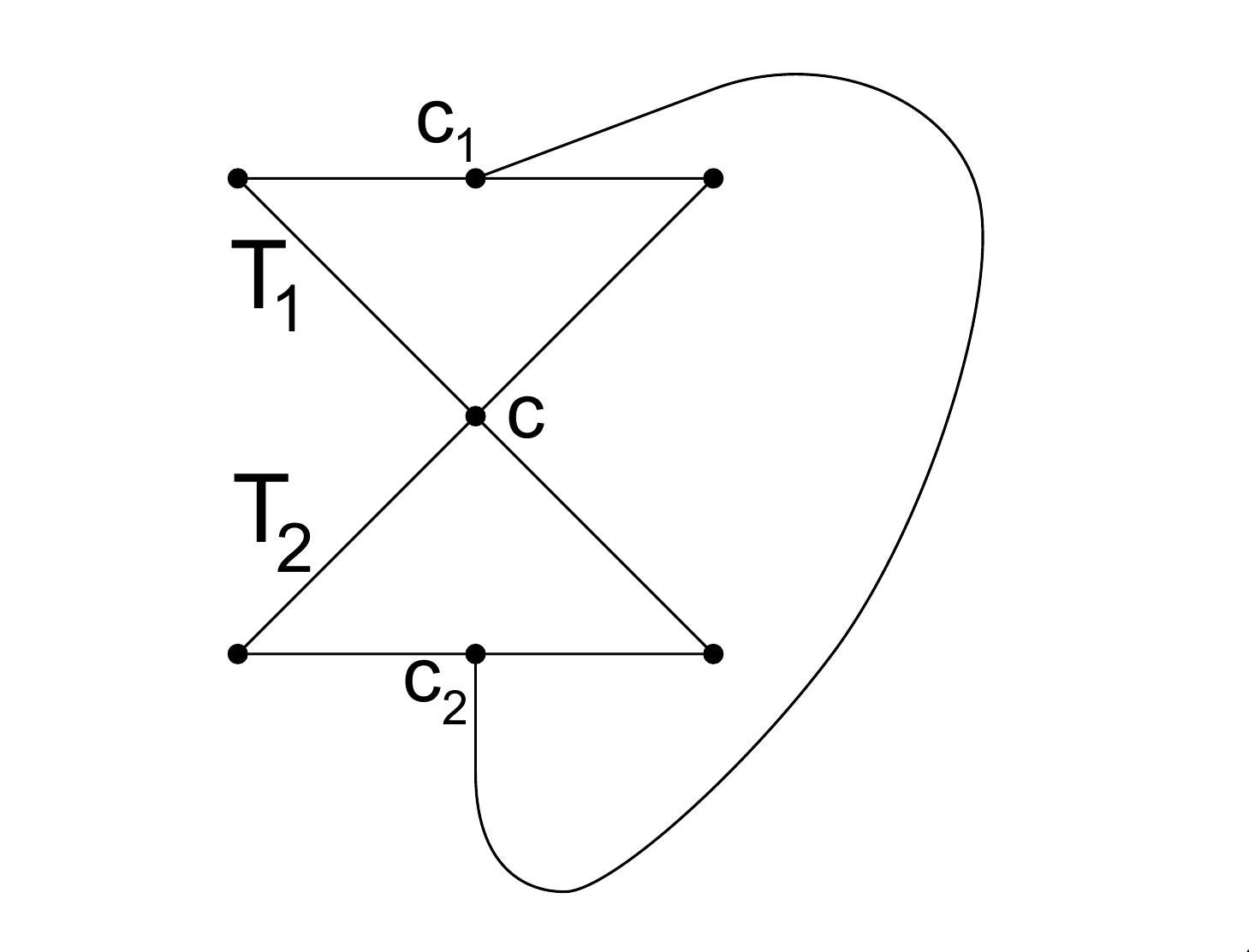}
	\caption{$n = 7$.}
	\label{figurethirteen}
\end{figure}

Without loss of generality, $e_{G}(a, T) = 0$. Keeping previous notations, let $k_{b}$ be the number of neighbours of $b$ in the cycle $T$. Consider $G(H)$. Since the vertex $a$ does not have neighbours among vertices of the cycle $T$ and the vertex $b$ have $k_{b}$ such neighbours, $d_{G(H)}(a) \geqslant \frac{n}{3}$ and $d_{G(H)}(b) \geqslant \frac{n}{3} - k_{b}$. The graph $G(H)$ contains a Hamiltonian path $H$. Recall that $|V(H)| = n - r$. Applying Lemma~\ref{lemma} to the path $H$: if $d_{G(H)}(a) + d_{G(H)}(b) \geqslant n - r$ then the graph $G(H)$ is Hamiltonian (that is the desired property). Consequently, $n - r > d_{G(H)}(a) + d_{G(H)}(b) \geqslant \frac{n}{3} + \frac{n}{3} - k_{b} = \frac{2n}{3} - k_{b} \Rightarrow k_{b} > r - \frac{n}{3}$. By Claim~\ref{claimthree} (item a), $k_{b} \leqslant \frac{r}{3}$. Thus, the following inequalities are deduced: $k_{b} \leqslant \frac{r}{3}$ and $k_{b} > r - \frac{n}{3} \Rightarrow \frac{r}{3} > r - \frac{n}{3} \Rightarrow n > 2r$, contrary to Claim~\ref{claimtwo}. 

\end{proof}

$ $

By Claim~\ref{claimfour}, $G(V(H))$ has a Hamiltonian cycle. From now on we will call $H$ a cycle.

$ $

Consider all pairs $(p, q)$ of distinct vertices of the cycle $H$ such that $pp_{0}, qq_{0} \in E(G)$ for some distinct $p_{0}, q_{0} \in V(T)$ (there exist such pairs or a pair of vertices because $G$ is 2-connected). Let us fix a pair $(p, q)$ such that the distance between $p$ and $q$ by the edges of the cycle $H$ is the shortest among such pairs. Let $p_{0}, q_{0} \in V(T)$ be distinct ($p_{0} \neq q_{0}$) neighbours of $p$ and $q$ respectively.

Let $H_{1}$ and $H_{2}$ be two sections of the cycle $T$ between $p$ and $q$ (see figure~\ref{figurefourteen}). Without loss of generality, $|V(H_{1})| \leqslant |V(H_{2})|$. Note that, for any vertex $f \in V(H_{1})$, inequality $d_{G(V(H))}(f) \geqslant \frac{n}{3}$ holds. Indeed, assume the converse. Since $d_{G}(f) \geqslant \frac{n}{3}$, $f$ has a neighbour (let it be $f_{1}$) in $V(T)$. Recall that $p_{0} \neq q_{0}$. Without loss of generality, $f_{1} \neq p_{0}$. Then $f$ and $p$ have two distinct neighbours in $T$, and the path between $f$ and $p$ by the edges of the cycle $H$ is shorter then the similar path between $p$ and $q$, contrary to the choice of $p$ and $q$.

%Hence, the path between $p$ and $q$ is not the shortest path on the edges of the cycle $H$ among all the pairs of the vertices of $H$ which have distinct neighbours in $T$, because path between $f$ and $p$ on the edges of the cycle $H$ is shorter.

\begin{figure}[htb]
	\centering
	\includegraphics[width=0.4\columnwidth, keepaspectratio]{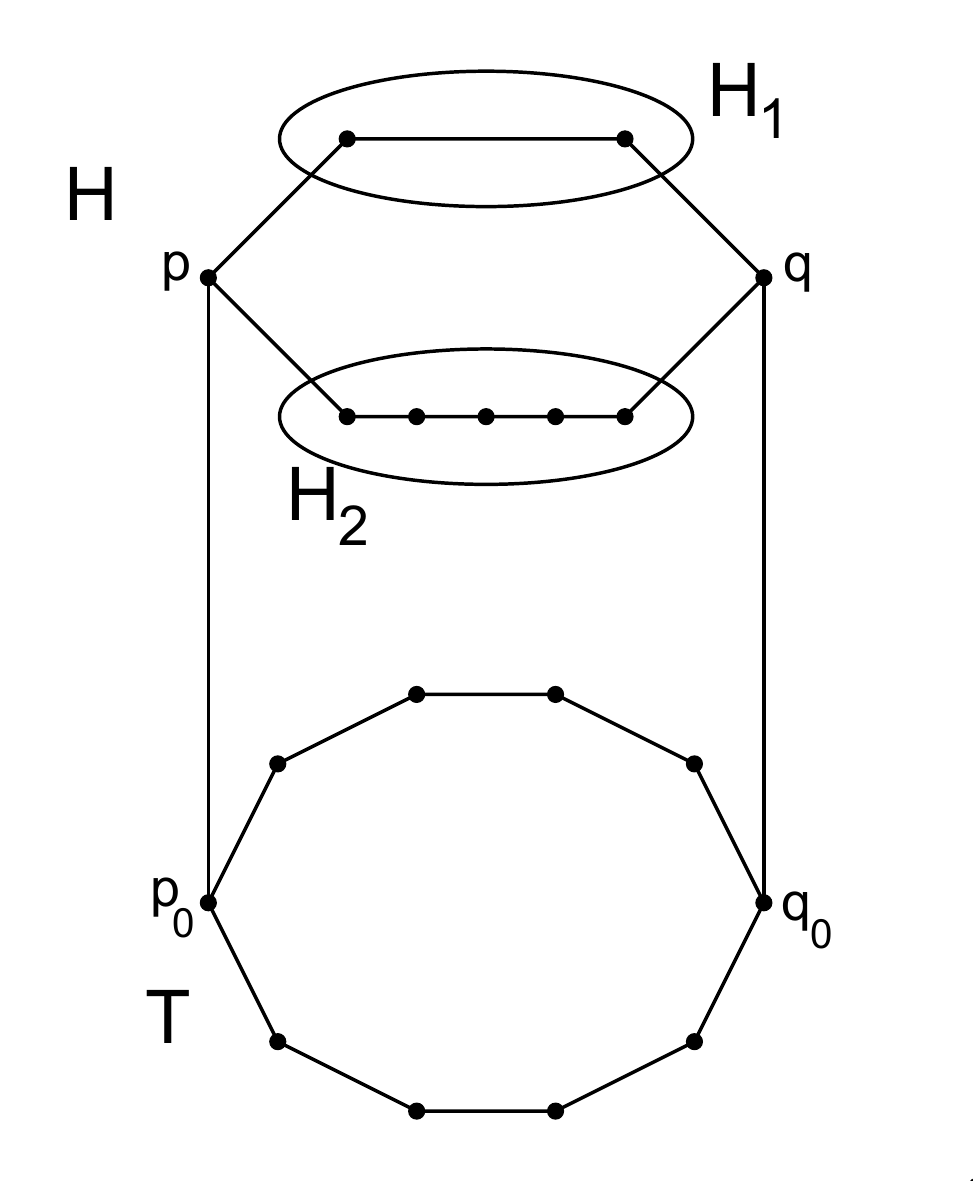}
	\caption{Partition of the cycle $H$ into $H_{1}$ and $H_{2}$; edges $pp_{0}$ and $qq_{0}$, where $p_0$ and $q_{0}$ are distinct vertices of the cycle $T$.}
	\label{figurefourteen}
\end{figure}

\begin{claim} \label{claimfive} The graph $G(V(H))$ has a Hamiltonian path with endpoints $p$ and $q$. 
\end{claim}

%$ $
 
%\textbf{Claim 5:} The graph $G(V(H))$ has such a Hamiltonian path that its endpoints are $p$ and $q$. 

%$ $

\begin{proof}

Consider the longest path with endpoints $p$ and $q$ (let it be $L_{2}$ and let it be $\eta$ vertices in it) among all such paths that if some vertex $d \in V(H)$ does not belong to this path then $d_{G(V(H))}(d) \geq \frac{n}{3}$, and all the vertices of the cycle $H$, which do not belong to this path, form a path (let this path be $L_{1}$ with endpoints $p_{1}$ and $q_{1}$). Such a path $L_{2}$ exists, because it follows from the proved before that the path $pH_{2}q$ is suitable (it has at least $\frac{n - r}{2} + 1$ vertices and hence $\eta \geq \frac{n - r}{2} + 1$). Assuming the converse, let $L_{2}$ be not a Hamiltonian path. Then the path $L_{1}$ contains $n - r - \eta \geq 1$ vertices. By the condition on $L_{2}$, $d_{G(V(H))}(p_{1}) \geq \frac{n}{3}$. Therefore, there are at least $\frac{n}{3} - (n - r - \eta - 1) = r + \eta + 1 - \frac{2n}{3}$ neighbours of the vertex $p_{1}$ among vertices of the path $L_{2}$. Note that neighbours of the vertex $p_{1}$ cannot be neighbours in the order of passing the path $L_{2}$ (otherwise, the vertex $p_{1}$ can be 'inserted' between these neighbours into the path $L_{2}$ (see figure~\ref{figurefifteen}), and then the path $L_{2}$ is longer, and all the remaining vertices form a path, and their degrees in $G(V(H))$ are at least $\frac{n}{3}$, contrary to the choice of the path $L_{2}$). Thus, $\eta \geq 2(r + \eta + 1 - \frac{2n}{3}) - 1 \Rightarrow \eta \leq \frac{4n}{3} - 2r - 1$.

\begin{figure}[htb]
	\centering
	\includegraphics[width=0.9\columnwidth, keepaspectratio]{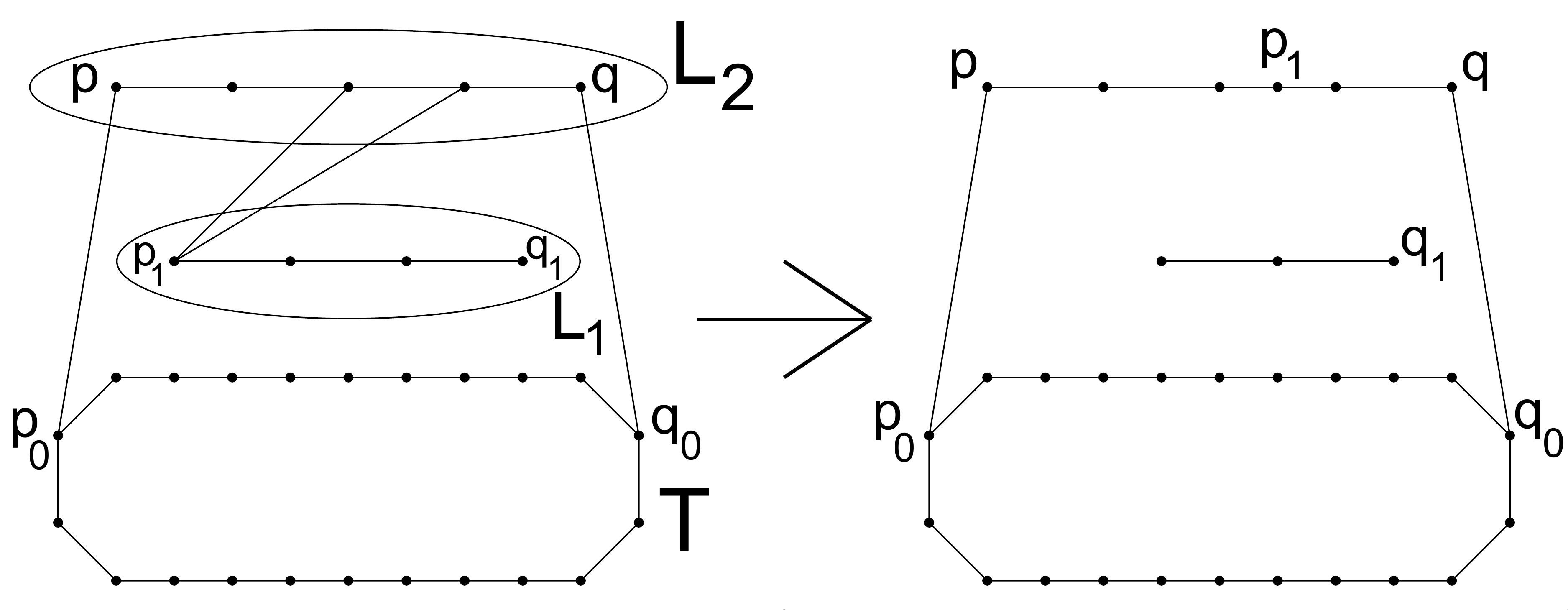}
	\caption{Partition of the cycle $H$ into two paths: $L_{1}$ and $L_{2}$. Demonstration that $p_{1}$ cannot have two consecutive (in the order of passing the path $L_{2}$) neighbours in $L_{2}$.}
	\label{figurefifteen}
\end{figure}

We need to prove that $G(V(L_{1}))$ has a Hamiltonian cycle. Assuming the converse, $G(V(L_{1}))$ does not have a Hamiltonian cycle. Firstly, let us consider the case, where $L_{1}$ contains at least 3 vertices. Since it has $n - r - \eta$ vertices and $L_{1}$ is a path, by Lemma~\ref{lemma}, $d_{G(L_{1})}(p_{1}) + d_{G(L_{1})}(q_{1}) < n - r - \eta$. Without loss of generality, $d_{G(L_{1})}(p_{1}) < \frac{n - r - \eta}{2}$. Then there are at least $\frac{n}{3} - \frac{n - r - \eta - 1}{2} = \frac{r + \eta + 1}{2} - \frac{n}{6}$ neighbours of the vertex $p_{1}$ among vertices of the path $L_{2}$. Recall that (see figure~\ref{figurefifteen}) neighbours of $p_{1}$ cannot be neighbours in the order of passing the path $L_{2}$, and hence $\eta \geq 2(\frac{r + \eta + 1}{2} - \frac{n}{6}) - 1 \Rightarrow \eta \geq r + \eta - \frac{n}{3} \Rightarrow \frac{n}{3} \geq r$, contrary to Claim~\ref{claimtwo}.

So, it remains to consider the case, where $L_{1}$ consists of one vertex or two vertices. Then $\eta \geq n - r - 2$, and since $\eta \leq \frac{4n}{3} - 2r - 1$, we have that $\frac{4n}{3} - 2r - 1 \geq n - r - 2 \Rightarrow r \leq \frac{n}{3} + 1$. By Claim~\ref{claimtwo}, $r \geq \frac{n}{2}$. Then $\frac{n}{3} + 1 \geq \frac{n}{2} \Rightarrow n \leq 6$. Then, by (1), $n = 6$, and hence all the inequalities in this paragraph are actually equalities. Thus, $n = 6$, $r = \frac{n}{2} = 3$, $\eta = n - r - 2 = 1$. But $\eta \geq 2$ because $L_{2}$ includes distinct vertices $p$ and $q$, a contradiction.

Consequently, $G(V(L_{1}))$ has a Hamiltonian cycle. From now on we will call $L_{1}$ a cycle.

\begin{figure}[htb]
	\centering
	\includegraphics[width=0.75\columnwidth, keepaspectratio]{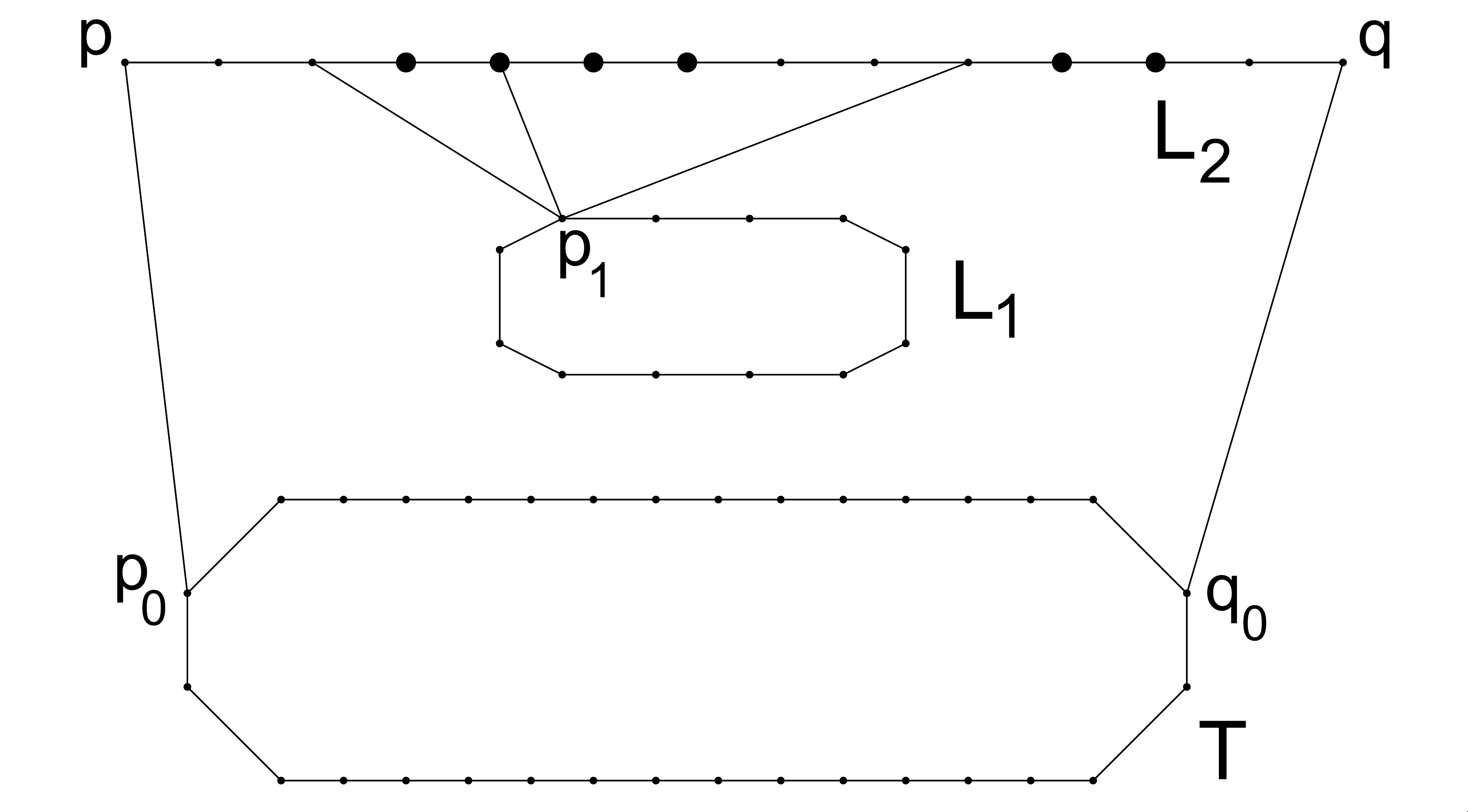}
	\caption{For each neighbour of $p_{1}$ we \textit{mark} the following two vertices in accordance with the fixed order of passing the path $L_{2}$.}
	\label{figuresixteen}
\end{figure}

We need to prove that there exist 2 vertices of the path $L_{2}$ such that one of them is adjacent to $p_{1}$, the other one is adjacent to $q_{1}$, and they are neighbours in the path $L_{2}$ or they have exactly one vertex between them in the order of passing the path $L_{2}$. Assuming the converse, they do not exist. We fix some direction of the path $L_{2}$ (for instance, from $p$ to $q$). Recall that there are at least $r + \eta + 1 - \frac{2n}{3}$ neighbours of $p_{1}$ among vertices of the path $L_{2}$. For each of them, we \textit{mark} the following two vertices in accordance with the fixed order of passing the path $L_{2}$ (see figure~\ref{figuresixteen}). Since it is proved (figure~\ref{figurefifteen} is dedicated to that) that all the neighbours of the vertex $p_{1}$ are not neighbours in the order of passing the path $L_{2}$, no vertex is marked twice. Hence, there are at least $2(r + \eta + 1 - \frac{2n}{3} - 1)$ marked vertices (for $q$ we do not mark anything, and hence we lose two vertices, and for the neighbour of $q$ in the path $L_{2}$ we mark only one vertex, and hence we lose one vertex, and therefore we lose at most two vertices because $q$ and the neighbour of $q$ in the path $L_{2}$ cannot be neighbours of the vertex $p_{1}$ simultaneously). Since we assumed the converse, all the neighbours of $q_{1}$ in the path $L_{2}$ (there are at least $r + \eta + 1 - \frac{2n}{3}$ such neighbours) cannot be marked. Thus, $$\eta \geq r + \eta + 1 - \frac{2n}{3} + 2(r + \eta + 1 - \frac{2n}{3} - 1) \Rightarrow $$

$$\Rightarrow 0 \geq r + 1 - \frac{2n}{3} + 2r + 2\eta - \frac{4n}{3} \Rightarrow 2n \geq 3r + 2\eta + 1\textit{.}$$ Since $\eta \geq \frac{n - r}{2} + 1$, we have $2n \geq 3r + n - r + 2 + 1 \Rightarrow n \geq 2r + 3 \Rightarrow r < \frac{n}{2}$, contrary to Claim~\ref{claimtwo}.

\begin{figure}[htb]
	\centering
	\includegraphics[width=0.55\columnwidth, keepaspectratio]{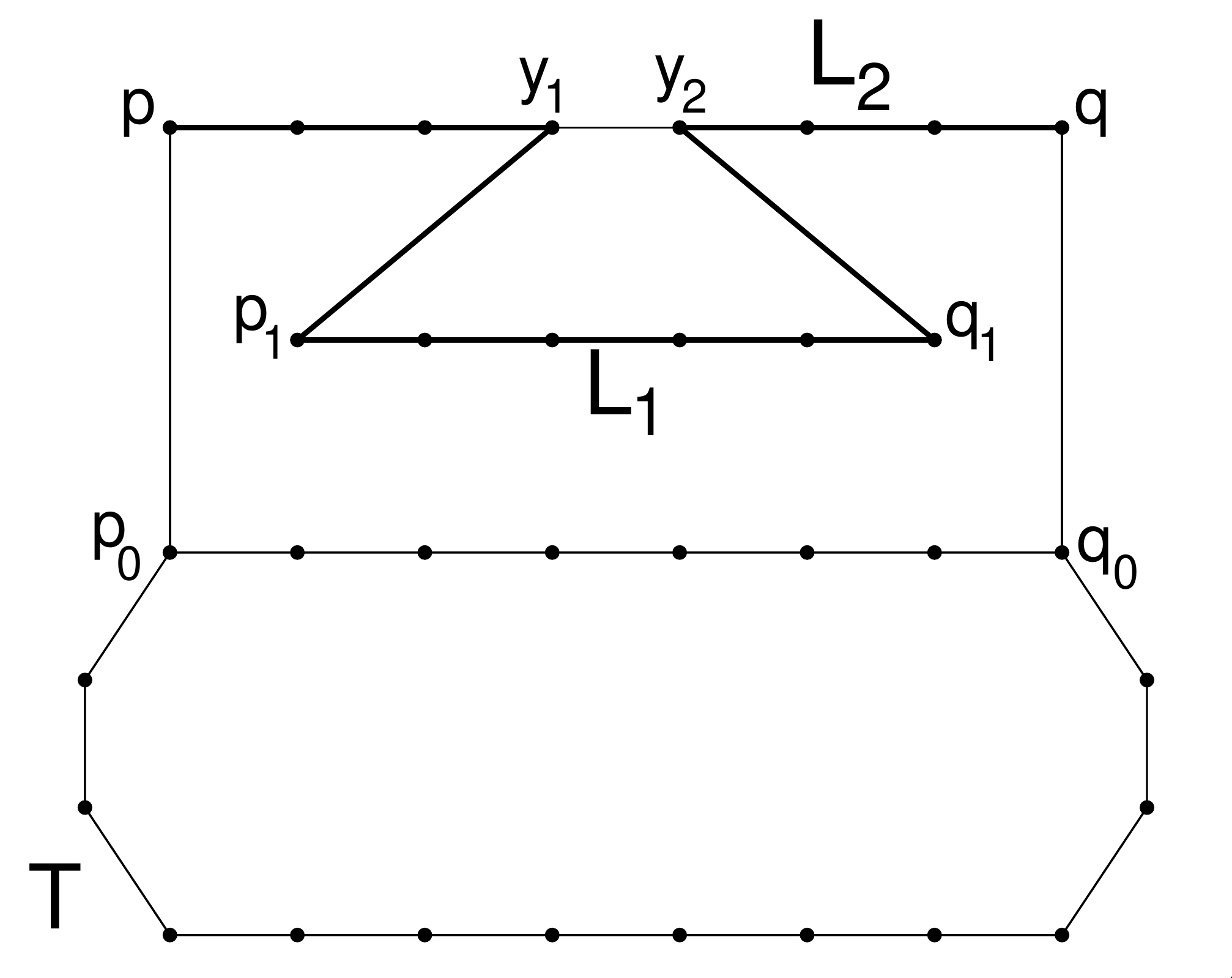}
	\caption{There is a Hamiltonian path in $G(V(H))$ with endpoints $p$ and $q$ if $y_{1}$ and $y_{2}$ are neighbours in the order of passing the path $L_{2}$.}
	\label{figureseventeen}
\end{figure}

Consequently, there exist 2 vertices of the path $L_{2}$ such that one of them is adjacent to $p_{1}$, the other one is adjacent to $q_{1}$ and they are neighbours in the path $L_{2}$ or they have exactly one vertex between them in the order of passing the path $L_{2}$. Let them be $y_{1}$ and $y_{2}$, let $y_{1}$ be closer to the vertex $p$ in the order of passing the path $L_{2}$, and $y_{2}$ - closer to $q$, and, without loss of generality, let $y_{1}$ be adjacent to $p_{1}$, $y_{2}$ be adjacent to $q_{1}$. If $y_{1}$ and $y_{2}$ are neighbours in the order of passing of the path $L_{2}$ then there exists a Hamiltonian path in $V(H)$ with endpoints $p$ and $q$ (see figure~\ref{figureseventeen}).

Therefore, there is a vertex (we denote it by $z$) between $y_{1}$ and $y_{2}$ in the path $L_{2}$. Since $z$ is a neighbour of both $y_{1}$ and $y_{2}$, as it was already mentioned (figure~\ref{figurefifteen} is dedicated to that), $e_{G}(z, \{p_{1}, q_{1}\}) = 0$. Assume that $e_{G}(z, L_{1}) \geqslant 1$. Then note that there is a path with endpoints $p$ and $q$ such that this path includes $V(L_{2})$ and at least one another vertex (it passes the path $L_{2}$ from $p$ to $z$, then it enters the cycle $L_{1}$, passes the edges of the cycle $L_{1}$ to the vertex $q_{1}$, goes to the vertex $y_{2}$ and then passes the edges of the path $L_{2}$ to the vertex $q$, see figure~\ref{figureeighteen}), a contradiction to the choice of the cycle $L_{2}$.

\begin{figure}[htb]
	\centering
	\includegraphics[width=0.5\columnwidth, keepaspectratio]{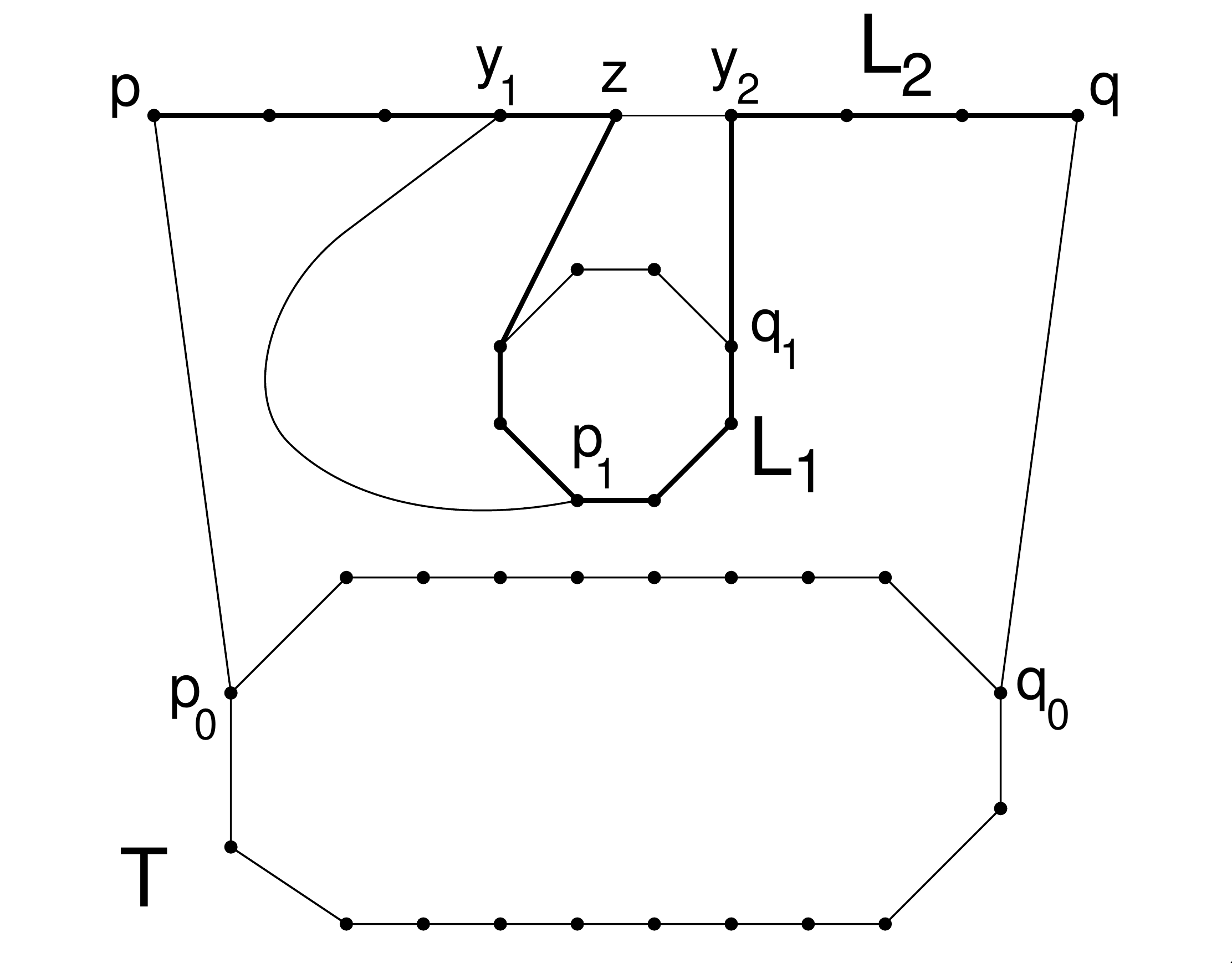}
	\caption{If $z$ has a neighbor from the cycle $L_{1}$ then there is a path (marked in bold) from $p$ to $q$ whose existence is contrary to the choice of $L_{2}$.}
	\label{figureeighteen}
\end{figure}

Hence, $e_{G}(z, L_{1}) = 0$. Suppose, $z$ has two consecutive neighbours in the order of passing the path $L_{2}$. Then there is the following Hamiltonian path in $V(H)$ with endpoints $p$ and $q$: it goes from $p$ to $y_{1}$ by the edges of the path $L_{2}$, it passes the edge $y_{1}p_{1}$, then it goes by the edges of the path $L_{1}$ to $q_{1}$, after that it enters the vertex $y_{2}$, and then it goes to $q$ by the edges of the path $L_{2}$ and, in so doing, we insert $z$ somewhere in this path between these neighbours (see figure~\ref{figurenineteen}).

Thus, $z$ does not have two consecutive neighbours in the order of passing the path $L_{2}$. Let it be $\tau_{1}$ vertices in the path $L_{2}$ between $p$ and $y_{1}$ (\textit{first section of the path $L_{2}$}) and let it be $\tau_{2}$ vertices in the path $L_{2}$ between $y_{2}$ and $q$ (\textit{second section of the path $L_{2}$}). Then $\tau_{1} + \tau_{2} = \eta - 1$. Since the vertex $z$ does not have two consecutive neighbours among vertices of the path $L_{2}$ in the order of passing the path $L_{2}$, it has at most $\frac{\tau_{1} + 1}{2}$ neighbours in the first section and at most $\frac{\tau_{2} + 1}{2}$ neighbours in the second section. Therefore, it has at most $\frac{\tau_{1} + 1}{2} + \frac{\tau_{2} + 1}{2} = 1 + \frac{\tau_{1} + \tau_{2}}{2} = 1 + \frac{\eta - 1}{2} = \frac{\eta + 1}{2}$ neighbours among vertices of the path $L_{2}$.

\begin{figure}[htb]
	\centering
	\includegraphics[width=0.5\columnwidth, keepaspectratio]{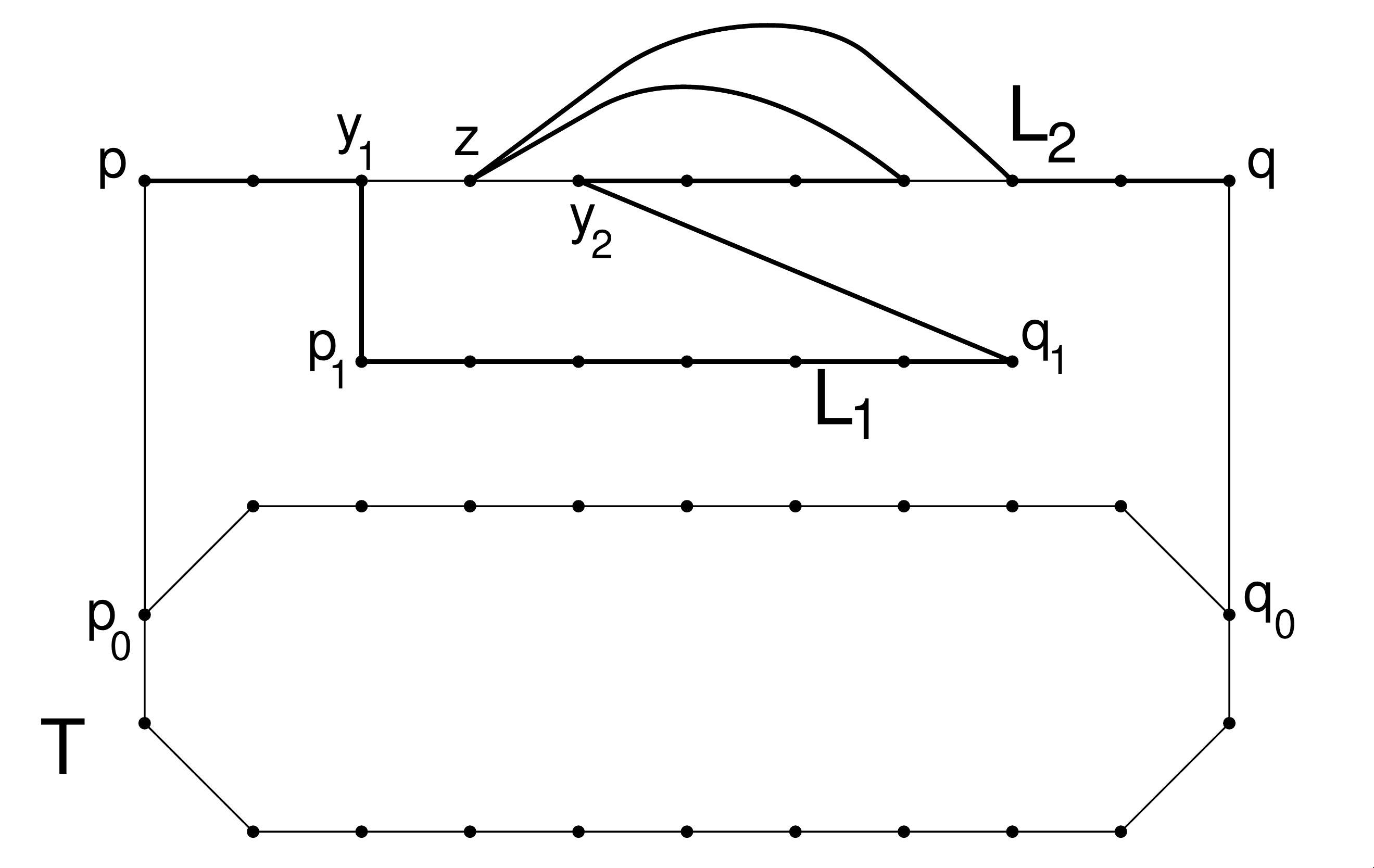}
	\caption{If $z$ has two consecutive neighbours in the order of passing the path $L_{2}$ then there is a Hamiltonian path in $V(H)$ (marked in bold).}
	\label{figurenineteen}
\end{figure}

By Claim~\ref{claimthree} (item b), $z$ has at most $\frac{r}{3}$ neighbours in the cycle $T$. Consequently, $d_{G}(z) \leq \frac{\eta + 1}{2} + \frac{r}{3}$. Then it follows from $d_{G}(z) \geq \frac{n + 2}{3}$ that $\frac{n + 2}{3} \leq \frac{\eta + 1}{2} + \frac{r}{3} \Rightarrow 2n + 4 \leq 3\eta + 3 + 2r \Rightarrow \eta \geq \frac{2n}{3} + \frac{1}{3} - \frac{2r}{3}$. Since $\eta \leq \frac{4n}{3} - 2r - 1$, we have that $\frac{4n}{3} - 2r - 1 \geq \frac{2n}{3} + \frac{1}{3} - \frac{2r}{3} \Rightarrow \frac{2n}{3} - \frac{4}{3} \geq \frac{4r}{3} \Rightarrow \frac{n}{2} > r$, contrary to Claim~\ref{claimtwo}.

\end{proof}

Finally, by Claim~\ref{claimfive}, there is a Hamiltonian path (we denote it by $H_{0}$) in the graph $G(V(H))$ with endpoints $p$ and $q$. Then there is a cycle: the path $H_{0}$, the edges $pp_{0}$ and $qq_{0}$ and the major half of the cycle $T$ between $p_{0}$ and $q_{0}$ (see figure~\ref{figuretwenty}). The length of this cycle is at least $\frac{r}{2} + 2 + n - r - 1 = n - \frac{r}{2} + 1$ (2 edges $pp_{0}$, $qq_{0}$ and the major half of the cycle $T$ have at least $\frac{r}{2} + 2$ edges and the Hamiltonian path in the graph $G(V(H))$; since there are exactly $n - r$ vertices, there are $n - r - 1$ edges in this path). Note that this is a cycle such that all other vertices beyond this cycle are one of the halves of the cycle $T$, that is, they form a path. Then it follows from the choice of the cycle $T$ that the following must be true: $r \geq n - \frac{r}{2} +  1 \Rightarrow \frac{3r}{2} > n \Rightarrow r > \frac{2n}{3}$, contrary to (2).  $\qedsymbol$

\begin{figure}[htb]
	\centering
	\includegraphics[width=0.4\columnwidth, keepaspectratio]{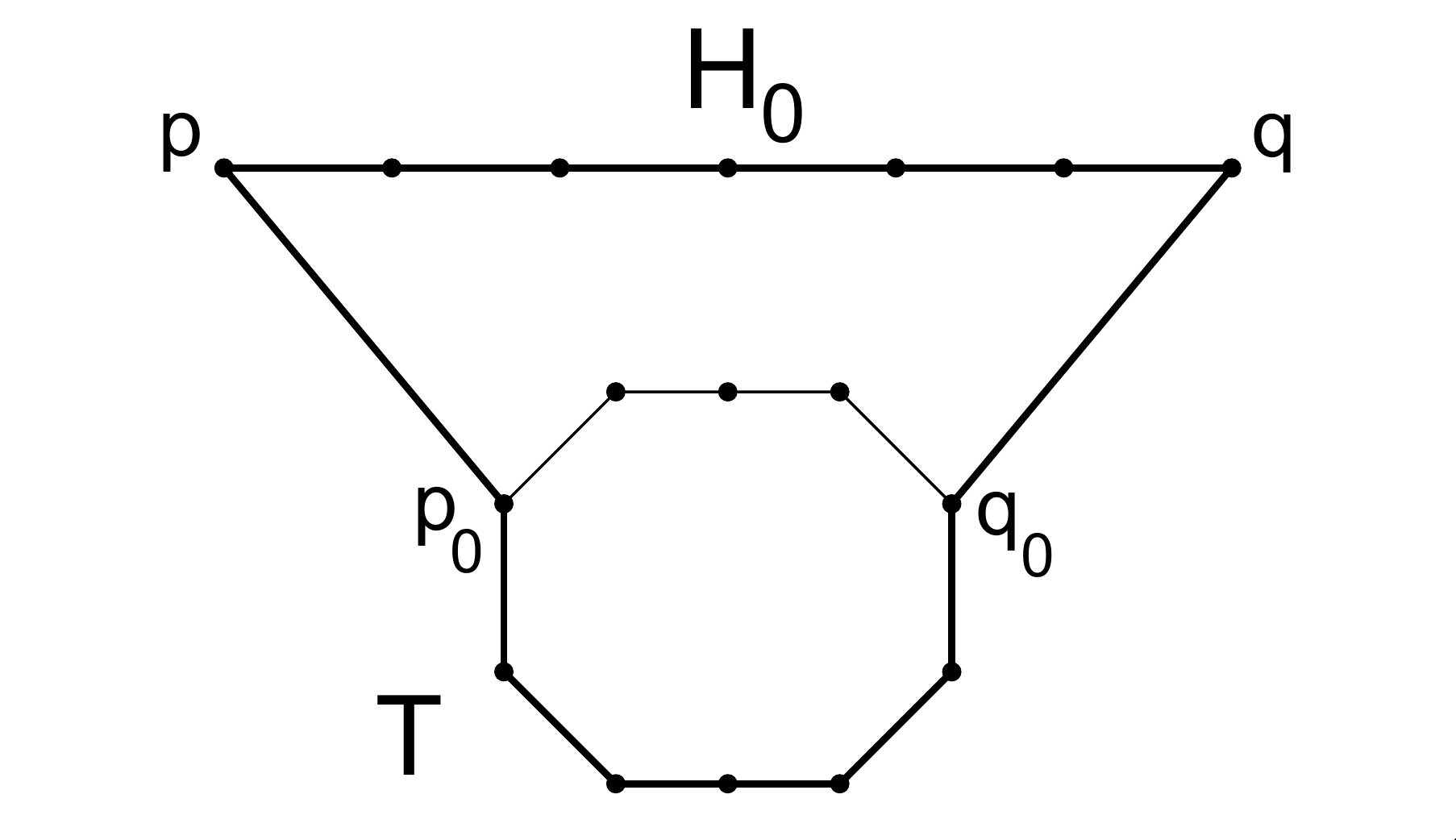}
	\caption{If there is a Hamiltonian path $H_{0}$ in the graph $G(V(H))$ with endpoints $p$ and $q$ then there is a cycle whose existence is contrary to the choice of $T$.}
	\label{figuretwenty}
\end{figure}

%\newpage

%$ $

%\newpage

$$ $$

\textbf{\Large{3 Our estimate cannot be improved}}

$ $

\begin{claim} \label{lastclaim} Our result is best possible in the sense that it becomes false if $\frac{n + 2}{3}$ is replaced by any smaller number. More specifically, for any $n \geq 8$ and for any $2 < \nu < \frac{n + 2}{3}$, there exists a 2-connected graph $G$ such that $v(G) = n$, $\delta(G) = \nu$ and in which there is no a cycle such that the vertex set beyond this cycle is independent.

\end{claim}

%\textbf{Claim:} Our $\frac{n + 2}{3}$ estimate is accurate. More specifically, for any $n \geq 8$ and for any $2 < \nu < \frac{n + 2}{3}$ there exists a 2-connected graph $G$ such that $v(G) = n$ and $\delta(G) = \nu$ and in which there is no such a cycle that the
%vertex set beyond this cycle is independent.

\begin{proof}

\begin{figure}[htb]
	\centering
	\includegraphics[width=0.65\columnwidth, keepaspectratio]{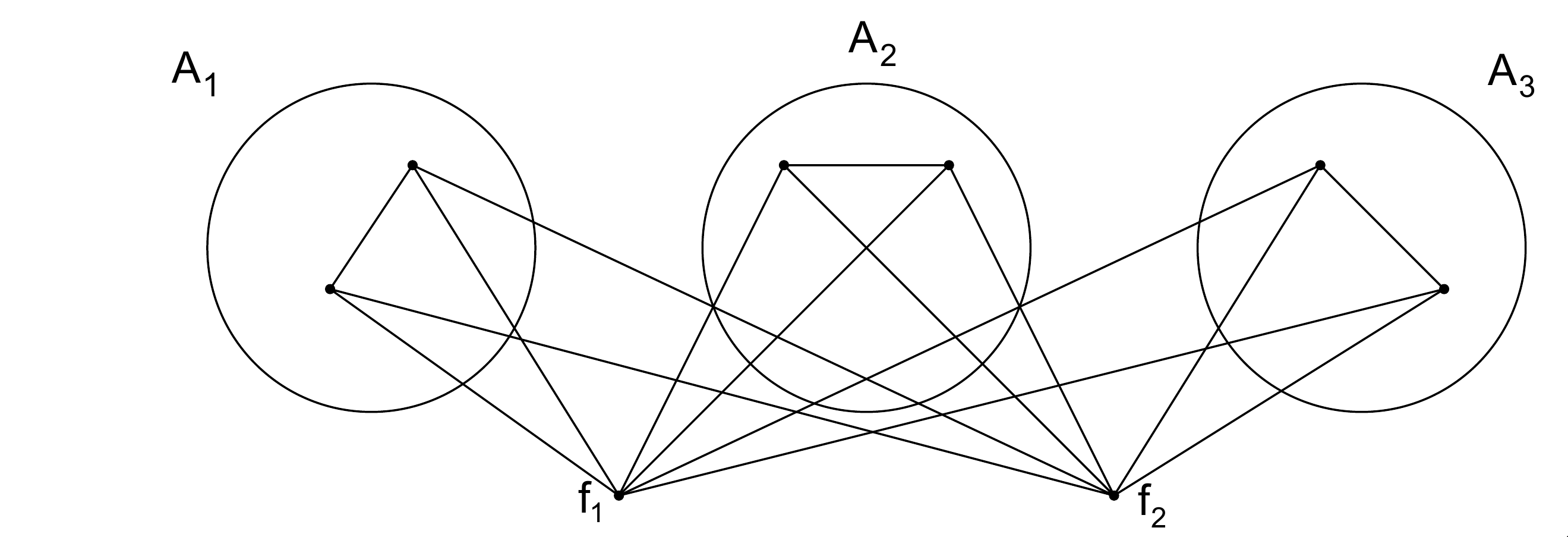}
	\caption{Our example for $n = 8$, $\nu = 3$.}
	\label{figuretwentyone}
\end{figure}

Let us take some 2 vertices, call them $f_{1}$ and $f_{2}$ and divide the remaining vertices into 3 groups: the first two groups contain exactly $\nu - 1$ vertices (let them be $A_{1}$ and $A_{2}$), and the third has $n - 2\nu$ vertices (let it be $A_{3}$). There are such edges in this graph: any two vertices in the same group are adjacent to each other, any two vertices in distinct groups are not adjacent to each other, and the vertices $f_{1}$ and $f_{2}$ are adjacent to all vertices, but not to each other (see figure~\ref{figuretwentyone}). Note that $\delta(G) = \nu$. Indeed, if $v \in A_{1} \cup A_{2}$ then $d_{G}(v) = \nu$, $d_{G}(f_{1}) = d_{G}(f_{2}) = n - 2 \geq \nu$ (since $n \geq 3$, $n - 2 \geq \frac{n}{3} > \nu$), and if $v \in A_{3}$ then $d_{G}(v) = n - 2\nu + 1 \geq \nu$.

Let us show that this graph is 2-connected. It is sufficient to show that $G - \{v\}$ is connected for any $v \in V(G)$. Without loss of generality, $v \neq f_{1}$. Clearly, $G - \{v\}$ is connected because all vertices in $(G - \{v\}) \setminus \{f_{2}\}$ are adjacent to $f_{1}$, and $f_{2}$ is adjacent to all vertices in $G \setminus \{v_{1}\}$.

%Indeed, when removing any (except $f_{1}$) vertex, any remaining vertex is adjacent to $f_{1}$ ($f_{2}$ is adjacent to some vertex from $A_{1}$ or $A_{2}$, which in turn is adjacent to $f_{1}$), thus the graph without the removed vertex is connected. If the vertex $f_{1}$ is removed, then any remaining vertex is connected to $f_{2}$, that is, the remaining graph is connected. 

It remains to prove that there is no a cycle in this graph such that the vertex set beyond this cycle is independent. Assuming the converse, there exists such a cycle. Then note that it must contain some vertices from all groups $A_{1}$, $A_{2}$, $A_{3}$ (otherwise, some entire group is not contained in the cycle, but any group has edges because $A_{1}$ and $A_{2}$ have $\nu - 1 > 1$ vertices, $A_{3}$ has $n - 2\nu > 1$ vertices). Thus, this cycle has vertices from all three groups. But, when this cycle goes from one group to another group, it must visit one of the vertices $f_{1}$, $f_{2}$. There are 3 groups $A_{1}$, $A_{2}$, $A_{3}$ and 2 vertices $f_{1}$, $f_{2}$, a contradiction.

%Let us start passing the edges of the cycle, starting from some vertex from $A_{1}$. Without loss of generality, moving along the edges of the cycle, at some point we will find ourselves at the vertex from $A_{2}$, then at some point at the vertex from $A_{3}$, and then we will return to the original position. Passing from the group $A_{i}$ to the group $A_{j}$ for $i \neq j$, the cycle must visit $f_{1}$ or $f_{2}$. But there are at least 3 of these transitions between groups, that is, $f_{1}$ or $f_{2}$ was visited by the cycle twice, a contradiction. 

\end{proof}

\newpage

\bibliographystyle{apa}

\begin{thebibliography}{99} 


\bibitem{bondy1} {\sc J.\,A.\,Bondy, \sc U.\,S.\,R.\,Murty.} {\it Graph Theory With Applications.}
Elsevier Science, 1976.


\bibitem{chartrandkapoor} {\sc G.\,Chartrand, \sc S.\,F.\,Kapoor.} {\it The cube of every connected graph is 1-Hamiltonian.} J. Res. Nat. Bur. Standards Sect. B 73 (1969), 47-48.



\bibitem{diestel2} {\sc R.\,Diestel.} {\it Graph Theory.}
Springer, 1997-2016 (editions 1-5).


\bibitem{linial3} {\sc N.\,Linial.} {\it A lower bound on the circumference of a graph.}
Discrete Math. v.15 (1976) p.297-300.



\bibitem{thomassen4} {\sc C.\,Thomassen.} {\it A theorem on paths in planar graphs.} Journal of Graph Theory, Vol.7 (1983) 169-176.

 

\end{thebibliography}

\end{document}